\documentclass[12pt]{article}
\usepackage{a4,amssymb,amsmath,amsthm}
\title{The graded algebra of Steenrod $q$th powers}
\author{Grant Walker}
\date{email: grant.walker@manchester.ac.uk\\ {\it \small 2010 Mathematics Subject Classification: 55S10}}

\newtheorem{theorem}{Theorem}[section]
\newtheorem{definition}[theorem]{Definition}
\newtheorem{example}[theorem]{Example}
\newtheorem{proposition}[theorem]{Proposition}
\newtheorem{remark}[theorem]{Remark}

\newcommand{\A}{\mathsf A}               
\newcommand{\bin}{{\rm bin}}        
\newcommand{\F}{{\mathbb F}_2}                    
\newcommand{\FF}{\mathbb F}        
\newcommand{\I}{\mathsf I}            
\newcommand{\len}{{\rm len}}           
\renewcommand{\P}{\mathsf P}             
\newcommand{\pin}{{\rm pin}}        
\newcommand{\U}{\mathsf U}           
\begin{document}
\maketitle

\begin{abstract}
The algebra $\A_q$ of Steenrod $q$th powers, where $q=p^e$ is a power of a prime $p$, is isomorphic to a subalgebra $\A'_q$ of the algebra of Steenrod $p$th powers $\A_p$.   The filtration of $\A_p$ by powers of its augmentation ideal was studied by J.~P.~May in his Princeton thesis of 1964.   We extend some of May's results to $\A_q$ and obtain a convenient set of defining relations for the graded algebra $E^0(\A_q)$.   In the case $q=p$, we recover the observation of S.~B.~Priddy that the subalgebra $E^0(\A_p(n-2))$ of $E^0(\A_p)$ generated by the elements $P^{p^j}$ for $0 \le j \le n-2$ is isomorphic to the graded algebra associated to the augmentation ideal filtration of the group algebra $\FF_p\U(n)$, where $\U(n)$ is the group of upper unitriangular matrices over $\FF_p$.

The Arnon A basis of $\A_p$ is given by monomials which are minimal in the left lexicographic order on formal monomials in the Steenrod powers.   K.~G.~Monks (for $p=2$) and  D.~Yu.~Emelyanov and Th.~Yu.~Popelensky (for $p >2$) have found a triangular relation between this basis and the Milnor basis using a certain ordering on the Milnor basis.   We introduce a variant of the Arnon A basis which is minimal for the right order, and show that this basis and Arnon's original A basis are also triangularly related to the Milnor basis of $\A_q$ using the right order on the Arnon A basis.
\end{abstract}

\section{Introduction} \label{sec_intro}

Given a prime $p$, we denote by $\A_p$ the algebra of Steenrod $p$th powers.   As an algebra over the field of $p$ elements $\FF_p$, $\A_p$ may be regarded as the subalgebra of the mod $p$ Steenrod algebra ${\cal A}_p$ which is generated by the elements $P^r$, $r \ge 0$, subject to the relation $P^0 =1$ and the Adem relations
$$
P^aP^b = \sum_{j=0}^{[a/p]} (-1)^{a+j} \binom{(p-1)(b-j)-1}{a-pj} P^{a+b-j}P^j, \ a<pb,
$$
where the binomial coefficients are taken mod $p$, or alternatively as the quotient algebra ${\cal A}_p/ {\cal A}_p\beta {\cal A}_p$, where $\beta$ is the Bockstein.   As a subalgebra of ${\cal A}_p$, the element $P^r$ is given the degree $2r(p-1)$, but for simplicity we regrade $\A_p$ by giving $P^r$ the `reduced' degree $r$.   Thus when $p=2$, $P^r$ will mean $Sq^r$, and not $Sq^{2r}$.

For a prime power $q =p^e$, where $e \ge 1$, the algebra $\A_q$ of Steenrod $q$th powers \cite[Chapter 11]{Smith} can be defined as an algebra over the Galois field $\FF_q$ by generators $P^r$, $r \ge 0$, subject to the relation $P^0 =1$ and the Adem relations
\begin{equation} \label{Adem_q}
P^aP^b = \sum_{j=0}^{[a/q]} (-1)^{a+j} \binom{(q-1)(b-j)-1}{a-qj} P^{a+b-j}P^j, \ a<qb,
\end{equation}
where, as before, the binomial coefficients are taken mod $p$.   As the coefficients lie in the prime subfield $\FF_p$, we have an algebra $\A'_q$ defined over $\FF_p$ by the same generators and relations, and an isomorphism $\rho: \A_q \rightarrow \A'_q \otimes_{\FF_p} \FF_q$ of Hopf algebras.

We introduce the algebras $\A_q$ and $\A'_q$ in Section \ref{sec_A'q}.   The algebra $\A'_q$ may be identified with the subalgebra of $\A_p$ with Milnor basis given by the elements $P(R) = P(r_1, r_2, \ldots)$ such that $r_j=0$ when $j \neq 0$ mod $e$ \cite[Section 12.3]{book}.   As for $\A_p$, we grade $\A_q$ by assigning degree $r$ to  $P^r$.   With this choice of gradings, the element $P^r \in \A_q$ corresponds to $P(0, \ldots, 0, r) \in \A'_q$, with $r$ in position $e$, and the map $\rho$ multiplies the grading by $(q-1)/(p-1)$.   It can be verified as in \cite[Proposition 3.2.1 or 12.3.3]{book} that the relations (\ref{Adem_q}) are satisfied in $\A'_q$.   As in \cite[Chapter 3]{book}, it follows from the Adem relations and the action of $\A'_q$ on the polynomial algebra $\FF_p[x]$ that the elements $P(0, \ldots, 0, p^s)$ for $s \ge 0$ are indecomposable in $\A'_q$, and form a minimal set of generators.

In Section \ref{sec_May} we discuss the filtration of $\A_p$ by powers of the augmentation ideal $\A_p^+$.   This was studied by J.~P.~May in his Princeton thesis of 1964 \cite{May}, and is known as the {\bf May filtration}.  Thus an element $\theta \in \A_p$ has May filtration $M(\theta) = m$ if $\theta \in  (\A_p^+)^m$  but  $\theta \not \in  (\A_p^+)^{m+1}$.   Since the ideal $\A_p^+$ is generated by the indecomposable elements $P^{p^j}$, $j \ge 0$, the elements of May filtration $1$ in degree $p^j$ are the elements whose expansion in any basis containing $P^{p^j}$ contains $P^{p^j}$ as a term.

\begin{example} \label{exa_22} {\rm
Since $Sq^2Sq^2 = Sq^3Sq^1 = Sq^1Sq^2Sq^1$, $M(Sq^2Sq^2) =3$.   Similarly $M(Sq^2Sq^1Sq^2) =3$ and $M(Sq^2Sq^1Sq^2Sq^1) = 4$.
}
\end{example}

\begin{definition} \label{def_alpha} {\rm
For an integer $a \ge 0$, let $a = \sum_{i=0}^r a_i p^i$ be the base $p$ expansion of $a$, where $0 \le a_i \le p-1$ for all $i$.  We define $\alpha(a) = \sum_{i=0}^r a_i$.
}
\end{definition}

The May filtration has some elementary properties.
\begin{proposition} \label{prop_elem_May}
{\rm (i)} If $\theta \in \A_p^d$, then $M(\theta) \ge \alpha(d)$.

{\rm (ii)} For all $\theta_1, \theta_2 \in \A_p$, $M(\theta_1 \theta_2) \ge M(\theta_1) + M(\theta_2)$.

{\rm (iii)} For all $\theta \in \A_p$, $M(\chi(\theta)) = M(\theta)$, where $\chi$ is the antipode of $\A_p$.
\end{proposition}

\begin{proof}
For (i), an element $\theta \in \A_p^d$ cannot be the product of $< \alpha(d)$ elements of the form $P^{p^j}$.   For (ii), we observe that $(\A_p^+)^r (\A_p^+)^s \subseteq (\A_p^+)^{r+s}$.  For (iii), we observe that $\chi(\A_p^+) = \A_p^+$.
\end{proof}

May determined the filtration on $\A_p$ by evaluating it on the Milnor basis, and showing that for a general element $\theta \in \A_p$, $M(\theta)$  is the minimum of the filtrations of the terms in the expansion of $\theta$ in the Milnor basis.

\begin{theorem} \label{prop_May} {\rm (May)}
For any sequence $R = (r_1, r_2, \ldots, r_\ell)$ of integers $\ge 0$, $M(P(R)) = \sum_{i=1}^\ell i\, \alpha(r_i)$.
\end{theorem}

We give a proof of Theorem \ref{prop_May} in Section \ref{sec_May}, and extend the result to the prime power case.

Given a filtered algebra $A$, we denote the associated graded algebra by $E^0(A)$.   The following observation of S.~B.~Priddy relates the Steenrod algebra to the group algebra of the group $\U(n)$ of $n \times n$ upper unitriangular matrices over $\FF_p$.    A proof of this result, based on \cite{May}, is given in \cite{Mitchell}.   The subalgebra $\A_p(n-2)$ of $\A_p$ is generated by the elements $P^{p^i}$, $0 \le i \le n-2$, and has dimension $p^{n(n-1)/2}$ as a vector space over $\FF_p$.

\begin{theorem} {\rm (Priddy)} \label{th_Priddy}
 The algebras $E^0(\FF_p \U(n))$ and $E^0(\A_p(n-2))$ are isomorphic as graded Hopf algebras.
\end{theorem}

In Section \ref{sec_graded alg} we give a convenient set of generators and relations for $E^0(\A_q)$, and in Section \ref{sec_Quillen} we use a theorem of Quillen \cite{Quillen} to prove Theorem \ref{th_Priddy}.   In Sections \ref{sec_Pst}, \ref{sec_Ar_A_p} and \ref{sec_Ar_A_q} we extend the work of K.~G.~Monks (for $p=2$) and of D.~Yu.~Emelyanov and Th.~Yu.~Popelensky (for $p >2$) on the relation between various bases of $\A_p$ \cite{Emelyanov-Popelensky17, Monks}.   In Section  \ref{sec_Pst} we show that some of the bases known as $P^s_t$-bases coincide with the Milnor basis up to higher May filtration, and hence give the same basis of $E^0(\A_q)$.  In Sections \ref{sec_Ar_A_p} and \ref{sec_Ar_A_q}, we introduce a variant of Arnon's A basis, and show that this basis and the original A basis are triangular with respect to the Milnor basis of $\A_q$, using the right lexicographic order on the A basis.   This is a different ordering from the one used in \cite{Emelyanov-Popelensky17, Monks}.

\section{The algebras $\A_q$ and $\A'_q$} \label{sec_A'q}

Let $p$ be a prime number and let $q = p^e$, $e \ge 1$, be a power of $p$.   The Milnor basis elements $P^r_e = P(0, \ldots, 0, r)$, with $r \ge 1$ in position $e$, generate a Hopf subalgebra $\A'_q$ of $\A_p$.   This subalgebra has an additive basis given by Milnor basis elements $P(R)$, where $R =(r_1, r_2, \ldots)$ is a finite sequence of integers $\ge 0$ such that $r_j =0$ if $j$ is not a multiple of $e$.   This is clear from Milnor's product formula, since a Milnor matrix $X = (x_{i,j})$ which arises when two such elements are multiplied must have entries $x_{i,j}=0$ unless $i$ and $j$ are both divisible by $e$, so that $i+j$ is divisible by $e$.   Hence we make the following definition.

\begin{definition} \label{def_P_e(R)} {\rm
For $e \ge 1$ and $R = (r_1, \ldots, r_\ell)$, let $R_e$ be the sequence whose $j$th term is $r_k$ if $j = ke$ for $k \ge 1$, and is $0$ otherwise, and let $P_e(R) = P(R_e)$ be the corresponding Milnor basis element of $\A_p$.   The algebra $\A'_q$ is the subalgebra of $\A_p$ with $\FF_p$-basis elements $P_e(R)$ for all finite sequences $R$ of integers $\ge 0$.
}
\end{definition}

\begin{proposition} \label{PseriesofA'q}
The Poincar\'e series $\Pi(\A'_q, \tau) =\prod_{j \ge 1} 1/(1-\tau^{(q^j-1)/(p-1)})$.
\end{proposition}

\begin{proof}
Since we use the reduced grading on $\A_p$ in which $P^r$ has degree $r$, $P(R)$ has degree $\sum_{j \ge 1} r_j(p^j-1)/(p-1)$, and $P_e(R)$ has degree  $\sum_{j \ge 1} r_j(q^j-1)/(p-1)$.
Thus $\dim (\A'_q)^d$ is the number of solutions $R = (r_1, r_2, \ldots)$ of the equation
\begin{equation} \label{degMilnor}
(p-1)d= (q-1)r_1 + (q^2-1)r_2 + \cdots + (q^j-1)r_j + \cdots,
\end{equation}
where $r_j \ge 0$ for $j \ge 1$.   A solution of (\ref{degMilnor}) gives an expression for $(p-1)d$ as the sum of $|R| = \sum_j r_j$ terms, of which $r_j$ are equal to $q^j-1$ for $j \ge 1$.   Since
 $1/(1 - \tau^{(q^j-1)/(p-1)}) = \sum_{i \ge 0} \tau^{i(q^j-1)/(p-1)}$, this corresponds to a term of degree $d$ in the power series expansion of $\prod_{j \ge 1} 1/(1 - \tau^{(q^j-1)/(p-1)})$.
\end{proof}

We show that $\A'_q$ satisfies the Adem relations (\ref{Adem_q}).

\begin{proposition} \label{prop_Ademq}
Let $q =p^e$ where $e \ge 1$, and for $r \ge 0$, let $P_e^r = P(0, \ldots, 0, r)$ in $\A_p$, where $r$ is in position $e$.   Then for $a < qb$
$$
P_e^a P_e^b = \sum_{j=0}^{[a/q]} (-1)^{a+j}\binom{(q-1)(b-j)-1}{a-qj} P_e^{a+b-j} P_e^j
$$
in $\A_p$, where the binomial coefficient is taken mod $p$.
\end{proposition}

\begin{proof}
By the Milnor product formula
$$
P_e^a P_e^b = \sum_{k=0}^b \binom{a+b-(q+1)k}{b-k} P_e(a+b-(q+1)k,k).
$$
Using this to expand both sides of the required relation in the Milnor basis, and equating coefficients of $P_e(a+b-(q+1)k,k)$, we see that the relation is equivalent to the identity
$$
\binom{a+b-(q+1)k}{b-k}  = \sum_{j=k}^{[a/q]} (-1)^{a+j}\binom{(q-1)(b-j)-1}{a-qj}\binom{a+b-(q+1)k}{j-k},
$$
of mod $p$ binomial coefficients, where $0 \le k \le b$.   The change of variables $a'=a-qk$, $b'=b-k$ and $j'=j-k$ reduces this to the case $k=0$, namely
\begin{equation} \label{eqn_Adem_2^r_1}
\binom{a'+b'}{b'}  = \sum_{j'=0}^{[a'/q]} (-1)^{a'+j'}\binom{(q-1)(b'-j')-1}{a'-qj'}\binom{a'+b'}{j'}.
\end{equation}
Replacing $a', b'$  by $a,b$, we prove (\ref{eqn_Adem_2^r_1}) by induction on $a+b$.   The base case $(a,b)=(0,0)$ holds since $\binom{c}{0}=1$ for all integers $c$.   We assume that the cases $(a-1,b)$ and $(a,b-1)$ hold, and prove the case $(a,b)$.   For this we use the identity
\begin{equation} \label{eqn_Adem_2^r_2}
\binom{c-1}{d} + \binom{c-1}{d-1} = \binom{c}{d} = \binom{c-q}{d} + \binom{c-q}{d-q}
\end{equation}
for mod $p$ binomial coefficients, which follows from $(1+x)^c = (1+x)^q (1+x)^{c-q}$, since $(1+x)^q = 1+x^q$ mod $p$.   Writing $c =(q-1)(b-j)$ and $d =a-qj$, we can expand each term on the right of (\ref{eqn_Adem_2^r_1}) in the form
\begin{eqnarray*}
\binom{c-1}{d}\binom{a+b}{j} &=& \binom{c-1}{d}\binom{a+b-1}{j} + \binom{c-1}{d}\binom{a+b-1}{j-1} \\
&=& -\binom{c-1}{d-1}\binom{a+b-1}{j} + \binom{c-q}{d}\binom{a+b-1}{j}  \\
&& +  \binom{c-q}{d-q}\binom{a+b-1}{j} + \binom{c-1}{d}\binom{a+b-1}{j-1}.
\end{eqnarray*}
The alternating sum over $j$ of the first term on the right is $\binom{(a-1)+b}{b}$ and that of the second term is $\binom{a+(b-1)}{b-1}$, by the inductive hypothesis.   Since $c-q = (q-1)(b-j-1)-1$ and $d-q = a-q(j+1)$, the third and fourth terms cancel on taking the alternating sums over $j$.   Since $\binom{a+b-1}{b}+ \binom{a+b-1}{b-1} = \binom{a+b}{b}$, this completes the inductive step.
\end{proof}

The Adem relations lead to a basis of admissible monomials for $\A'_q$.

\begin{definition} \label{def_adm_q} {\rm
For $A = (a_1, a_2, \ldots, a_\ell)$, where $a_i \ge 0$ for $1 \le i \le \ell$, $P_e^A = P_e^{a_1}P_e^{a_2} \cdots P_e^{a_\ell}$ is a {\bf monomial} in $\A'_q$.   The sequence $A= (a_1, \ldots, a_\ell)$ and the monomial $P_e^A$ are {\bf admissible} if $a_i \ge qa_{i+1}$ for $1 \le i < \ell$.
}
\end{definition}

If $a_i=0$ for all $i$ then $P_e^0 =1$ is admissible.  Otherwise it suffices to consider finite sequences $A$ of positive integers.  If $A = (a_1, a_2, \ldots, a_\ell)$, with $a_i >0$ for $1 \le i \le \ell$, we define the {\bf length} $\len(A) =\ell$ and the {\bf modulus} $|A| = \sum_{i=1}^\ell a_i$.   By adding trailing zeros, we may identify $A$ with the infinite sequence $(a_1, a_2, \ldots)$, where $a_i =0$ for $i>\ell$.   We introduce two linear orders on such sequences, the {\bf left order} $<_l$ and the {\bf right order} $<_r$.

\begin{definition}\label{def_lrorders} {\rm
Let $A = (a_1, a_2, \ldots)$ and $B = (b_1, b_2, \ldots)$ be sequences of integers $\ge 0$.   Then $A <_l B$ if and only if, for some $k$, $a_j=b_j$ for $1 \le j < k$ and $a_k < b_k$, and $A <_r B$ if and only if, for some $k$, $a_j =b_j$ for $j > k$ and $a_k > b_k$.
}
\end{definition}

Thus $<_l$ is the usual left lexicographic order, but $<_r$ is the reversal of the usual right lexicographic order.

\begin{proposition} \label{admis_span}
Every  element of $\A'_q$ is a sum of admissible monomials.
\end{proposition}

\begin{proof}
For $d \ge 0$, let $S^d$ denote the set of sequences $A = (a_1, a_2, \ldots)$ such that $a_i > 0$ for $i \le \len(A)$ and $|A| =d$.  Each element of $S^d$ gives a  corresponding monomial $P_e^A \in (\A'_q)^d$.   If $A$ is not admissible, then $0 < a_k < qa_{k+1}$ for some $k$.

Using the Adem relation (\ref{Adem_q}) with $a = a_k$ and $b=a_{k+1}$, we may write $P_e^A$ as a sum of monomials $P_e^B$, where $B$ is obtained from $A$ by replacing $(a,b)$ by $(a+b-j,j)$ for some $j$ such that $0 \le j \le [a/q] < b$.   In the case $j=0$, we omit the corresponding term of $B$: this does not affect $P^B$ since $P_e^0 =1$.   Then $B >_{l,r} A$.   Hence $P_e^A$ can be written as a sum of monomials $P_e^B$ which are greater than $P_e^A$ in both the left and right orders.   Iteration of this procedure must stop, since $S^d$ is a finite set.   Hence $P_e^A$ can be expressed as a sum of admissible monomials.
\end{proof}

\begin{definition} \label{def_Milnorvec} {\rm
Let $A = (a_1, \ldots, a_\ell)$  be an admissible sequence of length $\ell$.   The {\bf Milnor sequence} of $A$, or of $Sq^A$, is $R = (r_1, \ldots, r_\ell)$, where $r_j = a_j - qa_{j+1}$ for $1 \le j < \ell$ and $r_\ell = a_\ell$.
}
\end{definition}

\begin{proposition} \label{Adm_Mil}
The map  which sends an admissible sequence $A$ to its Milnor sequence $R$ is a bijection from the set of admissible sequences to the set of finite sequences of integers $\ge 0$, and it preserves the length $\ell$ and the right order $<_r$.   If $P^A_e \in (\A'_q)^d$, then $|R| = qa_1 - d$ and $d = \sum_{j=1}^\ell (q^j-1)r_j$.
\end{proposition}

\begin{proof}
Since $r_\ell =a_\ell$, the map $A \mapsto R$ preserves length.   Given $R = (r_1, \ldots, r_\ell)$ of length $\ell$, the linear equations $r_j = a_j - qa_{j+1}$ can be solved recursively for $j =\ell, \ell-1, \ldots, 1$ to give $a_j = \sum_{i=j}^\ell q^{i-j}r_i$.   In particular, $a_1 = \sum_{i=1}^\ell q^{i-1}r_i$.   Since $A = (a_1, a_2, \ldots)$ is admissible, these equations give an inverse map $R \mapsto A$.

For the right order, it suffices to consider sequences of the same length $\ell$, since $A <_r B$ when $\len(A) > \len(B)$.   Let $A = (a_1, \ldots, a_\ell)$ and $B = (b_1, \ldots, b_\ell)$ be admissible sequences of length $\ell$, with corresponding Milnor sequences $R =(r_1, \ldots, r_\ell)$  and $S =(s_1, \ldots, s_\ell)$.   If $a_j = b_j$ for $j > k$ and $a_k > b_k$, then $r_j = s_j$ for $j > k$ and $r_k > s_k$.   Hence $R <_r S$ if $A <_r B$.   The sum of the equations $r_j = a_j - qa_{j+1}$, $1 \le j < \ell$, gives $|R| = a_1 - (a_2 + \cdots + a_s) = qa_1-d$.   With the $j$th equation weighted by $q^{j-1}$, $d = \sum_{j=1}^\ell (q^j-1)r_j$.
\end{proof}

\begin{proposition}  \label{admissibles_p^e}
 The set of admissible monomials is a $\FF_p$-basis for $\A'_q$.
\end{proposition}

\begin{proof}
By Proposition \ref{admis_span}, the admissible monomials span $\A'_q$, and it follows from Proposition \ref{Adm_Mil} that there is a bijection between admissible monomials and Milnor basis  elements in $\A'_q$.
\end{proof}

Following \cite{book}, we denote by $\bin(a)$ the set of $2$-powers in the base $2$ decomposition of an integer $a \ge 0$.   For an odd prime $p$, this becomes a multiset.

\begin{definition} \label{def_pin} {\rm
For an integer $a \ge 0$, let $a = \sum_{i=0}^r a_i p^i$, where $0 \le a_i \le p-1$ for all $i$.  We denote by $\pin(a)$ the multiset of $\alpha(d)$ powers of $p$ whose sum is $d$.
}
\end{definition}

When $p=3$, for example, $\pin(25) = \{1,3,3,9,9\}$ and $\alpha(25) =5$.   Thus $\alpha(a)$ is the number of elements in $\pin(a)$, counted with multiplicity $\le p-1$.   The binomial coefficient $\binom{a}{b} \neq 0$ mod $p$ if and only if $\pin(b)$ is a sub-multiset of $\pin(a)$.

\begin{proposition}  \label{mingen_p^e}
The set of elements $P_e^{p^j}$, $j \ge 0$, is a minimal generating set for $\A'_q$ as an algebra over $\FF_p$.
\end{proposition}

\begin{proof}
Since $P_e^0=1$, $\A'_q$ is generated by the elements $P_e^k$, $k \ge 1$.   If $k$ is not a power of $p$, then $k = p^rs$, where $r \ge 0$ and $s \ge 1$.   Let $a = p^r$ and $b = p^r(s-1)$ in the Adem relation (\ref{Adem_q}).   We claim that $p^r \in \pin((p-1)b-1)$.   Since $(p-1)b -1 = -1$ mod $p^r$, this is equivalent to proving that $(p-1)b -1 \neq p^r-1$ mod $p^{r+1}$.  Since $b=p^r(s-1)$,  $(p-1)b-1 = p^r -1$ mod $p^{r+1}$ if and only if  $(s-1)(p-1) = 1$ mod $p$, and this is equivalent to $s(p-1) =0$ mod $p$.   Since $s \neq 0$ mod $p$, this proves the claim.

It follows that $\binom{(p-1)b-1}{a} \neq 0$ mod $p$, and so the last term $P_e^{a+b} = P_e^k$ appears in the Adem relation.   Hence $P_e^k$ is in the subalgebra generated by the elements $P_e^i$ for $i<k$.   By iterating the argument, it follows  that $\A'_q$ is generated by the elements $P_e^{p^j}$, $j \ge 0$.

If $P_e^{p^j}$ could be omitted from this generating set, then by considering the grading on $\A_q$ it would follow that $P_e^{p^j}$ is in the subalgebra generated by the elements $P_e^i$ for $i < p^j$.   But this is false, since for the action of $\A_p$ on $\FF_p[x]$, we have $P_e^i(x^{p^j}) = 0$ for $0 < i < p^j$, while $P_e^{p^j}(x^{p^j}) =  x^{qp^j}$.   Hence the generating set is minimal.
\end{proof}

The next result shows that $\A'_q$ is a Hopf subalgebra of $\A_p$.  The proof follows from the corresponding formulae in $\A_p$.

\begin{proposition} \label{coprod5_q}
The coproduct $\phi: \A'_q \rightarrow \A'_q \otimes  \A'_q$ satisfies

{\rm (i)} $\phi(P_e^k) = \sum_{i+j=k} P_e^i \otimes P_e^j$ for all $k \ge 0$,

{\rm (ii)} $\phi(P_e(R)) = \sum_{S+T=R} P_e(S) \otimes P_e(T)$ for all sequences $R$.
\qed
\end{proposition}

The dual Hopf algebra $\A_p^*$ of $\A_p$ is a polynomial algebra over $\FF_p$ with generators $\xi_j$, $j \ge 1$, of degree $p^j-1$ \cite{Milnor}.  The dual algebra $(\A'_q)^*$ is a quotient of $\A_p^*$, and may be described as follows.

\begin{proposition} \label{prop_dualq}
For $e \ge 1$ and $q=p^e$, $(\A'_q)^* = \A_p^*/I_q$, where $I_q$ is the Hopf ideal in $\A_p^*$ generated by the elements $\xi_j$ such that $e$ does not divide $j$.   The algebra
$(\A'_q)^*$ is a polynomial algebra over $\FF_p$ with generators $\xi_{ke}$, $k \ge 1$ of degree $q^k-1$.   The coproduct in $(\A'_q)^*$ is defined by $\phi(\xi_{ke}) = \sum_{i=0}^k \xi_{(k-i)e}^{q^i} \otimes \xi_{ie}$, and the antipode by $\chi(\xi_{ke}) = \sum_A \xi(A)$, where the sum is over compositions $A = (a_1, \ldots, a_s)$ of $k$, and $\xi(A) = \xi_{a_1e} \xi_{a_2e}^{q^{a_1}} \xi_{a_3e}^{q^{a_1+a_2}} \cdots \xi_{a_se}^{q^{a_1+ \cdots + a_{s-1}}}$.
\qed
\end{proposition}

We next introduce the algebra $\A_q$ of Steenrod $q$th powers \cite[Chapter 11]{Smith}.

\begin{definition} \label{def_Aq} {\rm
Let $q =p^e$ be a prime power, where $e \ge 1$.   The algebra $\A_q$ of Steenrod $q$th powers  is the algebra over the Galois field $\FF_q$ generated by elements $P^r$, $r \ge 0$, subject to the relation $P^0 =1$ and the Adem relations (\ref{Adem_q}).
}
\end{definition}

Note that the coefficients in the Adem relations lie in the prime subfield $\FF_p$, so that we also have an algebra defined over $\FF_p$ with the same generators and relations.   The preceding results allow us to identify this algebra with $\A'_q$.

\begin{proposition} \label{Aq_iso}
There is an isomorphism $\rho: \A_q \rightarrow \A'_q \otimes_{\FF_p} \FF_q$ of Hopf algebras defined by $\rho(P^r) = P_e^r$, $r \ge 0$.
\qed
\end{proposition}

As for $\A_p$, we grade $\A_q$ by assigning degree $r$ to $P^r$ for $r \ge 0$.   Since $P_e^r =P(0, \ldots, 0, r) \in \A_p$, with $r$ in position $e$, has degree $r(q-1)/(p-1)$, the map $\rho$ multiplies the grading by $(q-1)/(p-1)$.   The algebra $\A_q$ has a Milnor basis given by elements $P(R)$, where  $\rho(P(R)) = P_e(R)$.   The degree of $P(R)$ in $\A_q$ is $\sum_i r_i(q^i-1)/(q-1)$.   Then Propositions \ref{PseriesofA'q} and \ref{mingen_p^e} give the following result.

\begin{proposition} \label{PseriesofAq}
{\rm (i)} The elements $P^{p^j}$, $j \ge 0$, form a minimal generating set for $\A_q$ as an algebra over $\FF_q$.

{\rm (ii)} The Poincar\'e series of $\A_q$ is $\Pi(\A_q, \tau) =\prod_{j \ge 1} 1/(1-\tau^{(q^j-1)/(q-1)})$.
\qed
\end{proposition}

The Milnor product formula holds in $\A_q$ as in $\A_p$, with the row sum condition $r_i = \sum_{j \ge 0} p^j x_{i,j}$ on Milnor matrices replaced by $r_i = \sum_{j \ge 0} q^j x_{i,j}$, the column and diagonal sum conditions unchanged, and the multinomial coefficients taken in $\FF_p$.

The standard action of $\A_p$ on the polynomial algebra $\FF_p[x]$ is determined by the formula $P^k(x^d) = \binom{d}{k}x^{d+(p-1)k}$, so that $P^k(x^d)$ is a nonzero multiple of $x^{d+k}$ if $\pin(k) \subseteq \pin(d)$, and $P^k(x^d) =0$ otherwise.   It is usual to grade $\FF_p[x]$ by giving $x$ degree $2$ if $p>2$, but we assign grading $1$ to $x$ in all cases, as we deal only with $\A_p$ and not the full Steenrod algebra ${\cal A}_p$.   Thus the action of $P^k$ on a power of $x$ raises the degree by $k(p-1)$.   The Milnor basis element $P(0, \ldots, 0, k) \in \A'_q$, with $k$ in position $e$, acts on $\FF_p[x]$ by $P(0, \ldots, 0, k)(x^d) = \binom{d}{k}x^{d+(q-1)k}$.   We use the isomorphism $\rho$ of Proposition \ref{Aq_iso} to transfer this action to $\A_q$, so that (with notation in $\A_q$) the corresponding statement is $P^k(x^d) = \binom{d}{k}x^{d+(q-1)k}$.   This action of $\A_q$ in $\FF_p[x]$ can alternatively be defined directly, as in \cite[Chapter 1]{book}, by defining a `total Steenrod $q$th power' $\P = \sum_{k \ge 0} P^k$ acting on the polynomial algebra $\FF_p[x_1, \ldots, x_n]$ by the formulae $\P(1) =1$, $\P(x_i) =x_i + x_i^q$ for $1 \le i \le n$ and the Cartan formula $\P(fg) = \P(f)\P(g)$ for polynomials $f$ and $g$.

For $p=2$ and $k>0$, the combinatorial description of the formula $Sq^k(x^d) = x^{d+k}$ given in \cite[Proposition 6.1.2]{book} shows that the (reverse) base $2$ expansion of $d+k$ is obtained from that of $d$ by replacing at least one subsequence of the form $1 \cdots 1\ 0$ by a subsequence $0 \cdots 0\ 1$ of the same length.   Hence the number of digits $1$ in the sequence may be decreased, but not increased, and so $\alpha(d+k) \le \alpha(d)$.   A similar argument holds for an odd prime $p$, using the reversed base $p$ expansions of $d$ and $d+k(p-1)$.   We find that the expansion of $d+k(p-1)$ is obtained from that of $d$ by replacing at least one subsequence $a_1 \cdots a_{r-1}\ 0$, where $0 \le a_i \le p-1$, by a subsequence $0\ b_2 \cdots b_r$, with $\sum_i b_i \le \sum_i a_i$.   The same is true for the action of a general element $\theta \in \A_p$, since $\theta$ is a sum of compositions of the elements $P^k$.  The generalization to $\A_q$ behaves in the same way, using the mod $p$ expansions of $d$ and $d+k(p-1)$, since the action of $\A_q$ on $\FF_p[x_1, \ldots, x_n]$ is the same as that of $\A'_q$.

\section{The May filtration of $\A_q$} \label{sec_May}

The May filtration of $\A_q$ is defined as the filtration by powers of the augmentation ideal $\A_q^+$.   The following result extends  Theorem \ref{prop_May} to $\A_q$.

\begin{proposition} \label{May_filt_Mil_q}
For $R = (r_1, r_2, \ldots, r_\ell)$, the Milnor basis element $P(R) \in \A_q$ has May filtration $M(P(R)) = \sum_{i=1}^\ell i\, \alpha(r_i)$.   In particular, for $s \ge 0$ and $t \ge 1$ the element $P^s_t = P(0, \ldots, 0, p^s)$, where $p^s$ is in position $t$, has May filtration $t$.
\end{proposition}

The numerical function $\alpha$ is defined for $\A_q$ as for $\A_p$ (Definition \ref{def_pin}).   Thus if $a = \sum_{j \ge 0} a_j p^j$ where $0 \le a_j \le p-1$ for all $j$, then $\alpha(a) = \sum_{j \ge 0} a_j$, so that $\alpha(a)$ is the cardinality of the multiset $\pin(a)$ given by the base $p$ digits of $a$.   We note that the element $P_e(R) \in \A'_q \subset \A_p$ has May filtration $M(P_e(R)) = e \cdot M(P(R))$.

We shall establish Proposition \ref{May_filt_Mil_q} by proving inequalities in each direction.

\begin{proposition} \label{May_filt_Mil_1}
For $R = (r_1, r_2, \ldots, r_\ell)$, the Milnor basis element $P(R) \in \A_q$ has May filtration $M(P(R)) \ge \sum_{i=1}^\ell i\, \alpha(r_i)$.
\end{proposition}

\begin{proof}
We argue by induction on $|R| = \sum_{i=1}^\ell r_i$.   In the base case $|R|=1$, $R=(0, \ldots, 0,1)$ with $1$ in position $t$, say.   Thus the degree of $P(R)$ in $\A_q$ is $d=(q^t-1)/(q-1) = 1+q + q^2 + \cdots + q^{t-1}$.  Since an element of degree $d$ in $\A_q$ must have May filtration $\ge \alpha(d)$, $M(P(R)) \ge t$.   Thus we assume, as main induction hypothesis, that for $r \ge 2$ the inequality holds for all Milnor basis elements $P(S)$ in $\A_q$ with $|S| < r$, and we consider $P(R)$ with $|R| = r$.

For the inductive step, we first deal with the case where $R$ has only one nonzero term $r$, using a subsidiary induction on $\alpha(r)$.   If $\alpha(r) =1$, then $R=(0, \ldots, 0,p^s)$ with $p^s$ in position $t$, so that $P(R) = P^s_t$.  Thus $P(R)$ has degree $d = p^s(q^t-1)/(q-1)$ and $\alpha(d) = t$, so as before $M(P(R)) \ge t$.   For the inductive step on $\alpha(r)$, let $R=(0, \ldots, 0,r)$ where $\alpha(r) >1$.  Let $r=a+b$, where $a, b >0$ are chosen so that $\pin(r)$ is the disjoint union of $\pin(a)$ and $\pin(b)$.   Then in the Milnor product formula for $P(0, \ldots, 0, a) P(0, \ldots, 0, b)$ (both of length $t$) the initial Milnor matrix gives a nonzero multiple of $P(R)$.   Any other term $P(S)$ in the product arises from a Milnor matrix of the form
$$
X = \begin{array}{c|cc} &0 \cdots 0& b-c \\ \hline 0 & 0 \cdots 0&0  \\ \vdots & \vdots & \vdots  \\  0 & 0 \cdots 0 & 0 \\  a-q^tc & 0 \cdots 0 & c \end{array}\ ,\ {\rm where}\ c>0.
$$
Thus $S = (0, \ldots, 0, (a - q^tc) +  (b-c), 0, \ldots, 0, c)$, where $\pin(a - q^tc)$ and $\pin(b-c)$ are disjoint, and the nonzero terms are in positions $t$ and $2t$.   Then $|S| = r - q^tc < r$, and so by the main induction hypothesis on $|S|$, $M(P(S)) \ge t\alpha((a - q^tc) +  (b-c)) +2t\alpha(c)$.   Since $\alpha(a) \le \alpha(a - q^tc)+\alpha(q^tc) = \alpha(a - q^tc) +\alpha(c)$ and $\alpha(b) \le \alpha(b-c) + \alpha(c)$, $M(P(S)) \ge t(\alpha(a)+ \alpha(b)) = t\alpha(r)$.   By the induction hypothesis on $\alpha(r)$, we also have $M(P(0, \ldots, 0, a)) \ge t\alpha(a)$ and $M(P(0, \ldots, 0, b)) \ge t\alpha(b)$, so $M(P(0, \ldots, 0, a) P(0, \ldots, 0, b)) \ge t\alpha(r)$ also.   Hence the Milnor product formula implies that $M(P(R)) \ge t\alpha(r)$.   This completes the induction on  $\alpha(r)$, and the proof in the case $R =(0, \ldots, 0, r)$.

We may now assume that $R = (0, \ldots, 0, r_1, r_2, \ldots, r_\ell)$, where $r_1>0$ is in position $k$, and where $r_i>0$ for some $i>1$.   Applying the Milnor product formula to $P(0, \ldots, 0, r_2, \ldots, r_\ell) P(0, \ldots, 0, r_1)$, the initial Milnor matrix gives a nonzero multiple of $P(R)$.   Any other term $P(S)$ in the product arises from a Milnor matrix of the form
$$
X = \begin{array}{c|cc} & 0 \cdots 0 & r_1- \sum_{i=2}^\ell c_i \\ \hline 0 & 0 \cdots 0&0  \\ \vdots & \vdots & \vdots  \\  0 & 0 \cdots 0 & 0 \\ r_2-q^kc_2 & 0 \cdots 0 & c_2 \\  r_3-q^kc_3 & 0 \cdots 0 & c_3 \\ \vdots & \vdots & \vdots \\  r_\ell-q^kc_\ell & 0 \cdots 0 &  c_\ell \end{array}
$$
where the nonzero column is column $k$, the nonzero rows are rows $k+1, \ldots, k+\ell-1$, and some $c_i >0$.   Thus $|S| = |R| - q^k\sum_{i=2}^\ell c_i < |R|$, and so by the induction hypothesis on $|S|$
$$
M(P(S)) \ge k \alpha(r_1 -\sum_{i=2}^\ell c_i) +  \sum_{i=2}^\ell (k+i-1)\alpha(r_i - q^kc_i) +  \sum_{i=2}^\ell (2k+i-1)\alpha(c_i),
$$
where we have used the assumption that the coefficient $b(X) \neq 0$ to deal with diagonals in $X$ with two nonzero entries.

We use the inequalities
\begin{eqnarray*}
\alpha(r_1) &\le& \alpha(r_1 - \sum_{i=2}^\ell c_i) + \alpha(\sum_{i=2}^\ell c_i),\\
\alpha(r_i) &\le& \alpha(r_i - q^kc_i) + \alpha(c_i),\ 2 \le i \le \ell.
\end{eqnarray*}

Weighting and adding these inequalities, we obtain
\begin{eqnarray*}
k \alpha(r_1)  &+& \sum_{i=2}^\ell (k+i-1)\alpha(r_i)  \le k \alpha(r_1 - \sum_{i=2}^\ell c_i) +
k\alpha(\sum_{i=2}^\ell c_i)  \\ &+&  \sum_{i=2}^\ell (k+i-1) \alpha(r_i - q^kc_i) + \sum_{i=2}^\ell (k+i-1) \alpha(c_i) \\
&\le& M(P(S)).
\end{eqnarray*}
We conclude that if the term $P(S)$ appears in the expansion of
$$
P(0, \ldots, 0, r_2, \ldots, r_\ell)\, P(0, \ldots, 0,r_1)
$$
then $M(P(S)) \ge \sum_{i=1}^\ell (k+i-1)\alpha(r_i)$.   By the induction hypothesis, the product itself has May filtration $\ge \sum_{i=2}^\ell (k+i-1)\alpha(r_i) + k\alpha(r_1) = \sum_{i=1}^\ell (k+i-1)\alpha(r_i) $, and so the same is true for the initial term $P(R)$.   This completes the induction.
\end{proof}

To complete the proof of Theorem \ref{prop_May}, we use the standard action of Milnor basis elements on polynomials to prove the reverse inequality.

\begin{proposition} \label{May_filt_Mil_2}
For $R = (r_1, r_2, \ldots, r_\ell)$, the Milnor basis element $P(R) \in \A_q$ has May filtration $M(P(R)) \le \sum_{i=1}^\ell i\, \alpha(r_i)$.
\end{proposition}

\begin{proof}
We first consider two special cases.  For the first case, let $q = p$ and $P(R) = P^s_t$, so that $P^s_t (x^{p^s}) = x^{p^{s+t}} \in \FF_p[x]$.   There is only one digit $1$ in the (reversed) base $p$ expansion of $p^s$, which is moved from position $s$ to position $s+t$.   This move cannot be broken down into more than $t$ steps by the action of Steenrod operations in $\A_p$ of positive degree, and so $P^s_t$ cannot be expressed as an element of $(\A_p^+)^{t+1}$.   Hence $M(P^s_t) \le t$.   For a general $q =p^e$, the element $P^s_t \in \A_q$ corresponds via $\rho$ to  $P^s_{et} \in \A'_q$, and the same argument holds with $t$ replaced by $et$.

For the second case, consider the equation $P^r(x^r) = x^{pr}$.   Each of the $\alpha(r)$ nonzero digits of the base $p$ expansion of $r$ is moved one place to the right to give the base $p$ expansion of $pr$.  This cannot be done by a sequence of more than $\alpha(r)$ moves, and so $M(P^r) \le \alpha(r)$.    For a general $q =p^e$, we have $P^s_t (x^{q^s}) = x^{q^{s+t}}$, where $P^s_t = P(0, \ldots, 0, p^s)$ with $p^s$ is in position $t$.   Replacing the equation $P^r(x^r) = x^{pr}$ by $P^r(x^r) = x^{qr}$, the effect on the base $p$ expansion is to move all $\alpha(r)$ base $p$ digits of $r$ through $e$ places to the right to give the base $p$ expansion of $qr$.

For the general case, we combine these two examples, using the action of $\A_q$ on polynomials over $\FF_p$ in variables $x_1, x_2, \ldots$, and representing monomials by arrays called `blocks' as in \cite{book}, where the rows of the block are formed by the reverse binary expansions of the exponents.   Let $R = (r_1, r_2, \ldots, r_\ell)$.   Then $P(R) \in \A_q$ acts on a product $x_1x_2 \cdots x_{|R|}$ to give the sum of all the monomials obtained by raising $r_i$ of the variables $x_j$ to $x_j^{q^i}$ for $1 \le i \le \ell$.   By specializing the first $r_1$ variables to a new variable $y_1$, the next $r_2$ variables to a new variable $y_2$, and so on, we find that $P(R)(y_1^{r_1}y_2^{r_2} \cdots y_\ell^{r_\ell}) = y_1^{qr_1}y_2^{q^2 r_2} \cdots y_\ell^{q^\ell r_\ell} +$ other terms, since no other specialization of the variables leads to the same monomial.   This monomial $y_1^{qr_1}y_2^{q^2 r_2} \cdots y_\ell^{q^\ell r_\ell}$ has the same $\alpha$-count as the original monomial $y_1^{r_1}y_2^{r_2} \cdots y_\ell^{r_\ell}$, and its base $p$ block is obtained by moving each of the nonzero digits in row $i$ a total of $ei$ places to the right, for $1 \le i \le \ell$.   This cannot be achieved by a sequence of more than $\sum_{i=1}^\ell i\, \alpha(r_i)$ Steenrod operations of positive degree.   The result follows.
\end{proof}

\section{The graded algebra $E^0(\A_q)$} \label{sec_graded alg}

In this section we consider the graded algebra $E^0(\A_q)$ associated to the May filtration of $\A_q$.   Given a filtered algebra $A$, We use the same notation ambiguously for elements of $A$ and for the corresponding elements of the associated graded algebra $E^0(A)$.

Using Theorem \ref{prop_May}, the structure of $E^0(\A_q)$  can be described as follows.   The Milnor basis elements $P(R)$, $R = (r_1, r_2, \ldots)$ form a $\FF_q$-basis of $E^0(\A_q)$, with $P(R)$ in grading $M(P(R)) = \sum_{i=1}^\ell i\, \alpha(r_i)$.   The product in $E^0(\A_q)$ is given by $P(R)P(S) = \sum_T P(T)$, where this sum is obtained by deleting all terms $P(T)$ in the corresponding product in $\A_q$ such that $M(P(T)) > M(P(R)) + M(P(S))$.   For example, $Sq(4,2)Sq(1,2) = Sq(1,3,1) + Sq(4,2,1)$ in $\A_2$, but $Sq(4,2)Sq(1,2) = Sq(4,2,1)$ in $E^0(\A_2)$.   By the `degree' and `grading' of $\theta \in E^0(\A_q)$ we mean the degree of $\theta$ as an element of $\A_q$ and its filtration $M(\theta)$.

\begin{proposition} \label{prop_power_comm}
Let $\A_q(n-2)$ be the subalgebra of $\A_q$ generated by $P^{p^i}$, $0 \le i < (n-1)e$.   The graded algebra $E^0(\A_q(n-2))$ associated to the May filtration of $\A_q(n-2)$ can be defined by generators and relations as follows.

1.  There is one generator $P^s_t = P(0, \ldots, 0, p^s)$ in each {\bf $q$-atomic} degree
$$
a = p^s(q^t-1)/(q-1)
$$
with $s \ge 0$, $t \ge 1$ and $(s/e)+t < n$.  There are $en(n-1)/2$ such degrees $a$.   We denote this generator by $P[a]$, and if $a$ and $b$ are $q$-atomic we denote the commutator $P[a]P[b] - P[b]P[a]$ by $[P[a], P[b]]$.

2. The defining relations are the power relations $P[a]^p =0$ for all $a$, and the commutator relations: for $a>b$,
$$
[P[a], P[b]] =\begin{cases} P[a+b], & \text{if $a+b$ is $q$-atomic and is not a power of $p$},\\0, & \text{otherwise}. \end{cases}
$$
Thus a $\FF_q$-basis for $E^0(\A_q(n-2))$ is given by monomials with exponents $\le p-1$ in the generators $P[a]$, taken in a fixed but arbitrary order.   The dimension of the $\FF_q$-algebra $E^0(\A_q(n-2))$ is $p^{en(n-1)/2} =q^{n(n-1)/2}$.

\end{proposition}

The $q$-atomic number $a = p^s(q^t-1)/(q-1)$ has reversed base $p$ expansion $0 \ldots 0\ 1\ 0 \ldots 0\ 1 \ldots 0 \ldots 0\ 1$, where there are $t$ equally spaced digits $1$ in positions $s, s+e, \ldots, s+(t-1)e$.   By Proposition \ref{May_filt_Mil_q}, $P[a]$ has grading $t$ in $E^0(\A_q)$.

\begin{example} \label{exa_E0_1} {\rm
In the case $q=p$, $n=4$,  $E^0(\A_p(2))$ has $6$ generators $P[1] = P(1)$, $P[p]= P(p)$, $P[p+1] = P(0,1)$, $P[p^2] =P(p^2)$, $P[p^2+p]= P(0,p)$ and $P[p^2+p+1] =P(0,0,1)$.   The nonzero commutators are $[P[p],P[1]] =P[p+1]$, $[P[p^2],P[p]] =P[p^2+p]$ and $[P[p^2+p],P[1]] = [P[p^2],P[p+1]] = P[p^2+p+1]$.
}
\end{example}
\begin{example} \label{exa_E0_2} {\rm
In the case $q=4$, $n=3$, $E^0(\A_4(1))$ has $6$ generators $P[a]$, $a =1,2,4,8,5,10$ with relations $P[a]^2 =0$ and nonzero commutators $[P[4],P[1]] = P[5]$ and $[P[8],P[2]] = P[10]$.   The elements $P[a] \in \A_4$ correspond respectively to $Sq(0,1)$, $Sq(0,2)$, $Sq(0,4)$, $Sq(0,8)$, $Sq(0,0,0,1)$ and $Sq(0,0,0,2)$ in $\A'_4 \subset \A_2$.
}
\end{example}

These power and commutator relations do not generally hold in $\A_q$, but only in the graded algebra $E^0(\A_q)$.   For example, in $\A_4$ we have $[P[8],P[1]] = P[5]P[4]$, $[P[4],P[2]] =P[5]P[1]$ and $[P[8],P[4]] = P[10]P[2]$, where in each case the right hand side has filtration $3$, and $P[4]P[4] = P(3,1)$ and $P[8]P[8] = P(6,2)$, where in each case the right hand side has filtration $4$.

The particular power relations $(P^s_1)^p = (P^{p^s})^p = 0$ in $E^0(\A_q(n-2))$ can be proved by observing that the expansion of $(P^{p^s})^p$ in the admissible basis does not involve $P^{p^{s+1}}$.   This is because $P^{p^{s+1}}$ is indecomposable in $\A_q$ (Proposition \ref{PseriesofAq}(i)), and since $p^{s+1} = p^s + \cdots + p^s$ ($p$ terms) is the unique minimal decomposition of $p^{s+1}$ as the sum of more than one power of $p$, all other admissible monomials in degree $p^{s+1}$ have May filtration $\ge p+1$.   To prove the general power relation $(P^s_t)^p=0$ and the commutator relations, we use the Milnor product formula.

\begin{proof}[Proof of Proposition \ref{prop_power_comm}]
Given a $q$-atomic number $a = p^s(q^t-1)/(q-1)$, let $P[a] = P(0, \ldots, 0, p^s)$ be the corresponding element of $\A_q(n-2)$, where $p^s$ is in position $t$.   Similarly let $b = p^u(q^v-1)/(q-1)$ so that $P[b] = P(0, \ldots, 0, p^u)$ with $p^u$ in position $v$.   Then we wish to prove that in $E^0(\A_q(n-2))$ we have $P[a]^p =0$ for all $a$ and $[P[a],P[b]] = P[a+b]$ when $a \ge b$ and $s = u +v$ or $u =s+t$, and $[P[a],P[b]] = 0$ otherwise.   Equivalently, $P[a]^p$ has May filtration $> pt$, $[P[a],P[b]] - P[a+b]$ has May filtration $> t+v$ when $s = u +v$ or $u = s+t$, and otherwise $[P[a],P[b]]$ has May filtration $> t+v$.

The calculation of $P[a]P[b]$ and $P[b]P[a]$ by the Milnor product formula produces Milnor matrices of the form
$$
\begin{array}{c|cccc} &0 &\cdots & 0 & p^u -c \\ \hline 0 & 0& \cdots & 0 & 0 \\ \vdots & \vdots && \vdots & \vdots \\ 0  & 0& \cdots & 0 & 0 \\ p^s-p^v c & 0 & \cdots &0 & c \end{array}, \quad
\begin{array}{c|cccc} &0 &\cdots & 0 & p^s -c \\ \hline 0 & 0& \cdots & 0 & 0 \\ \vdots & \vdots && \vdots & \vdots \\ 0  & 0& \cdots & 0 & 0 \\ p^u-p^t c & 0 & \cdots &0 & c \end{array}
$$
We may assume that $t \ge v$, and that if $t=v$ then $s \ge u$.

We first consider the case $c>0$.   Since $c$ appears in  position $t+v$, by Proposition \ref{May_filt_Mil_1} the corresponding term in $P[a]P[b]$ has May filtration $\ge t+v$, and is $> t+v$ unless the array satisfies $\alpha(p^s-p^vc) +\alpha(c) = \alpha(p^s) =1$ and $\alpha(p^u-c) + \alpha(c) = \alpha(p^u) =1$.   Thus $p^u-c =0$ and $p^s-p^vc =0$, i.e. $s = u +v$, $c =p^u$.   The corresponding element is $P[a+b] = P(0, \ldots, 0, p^u)$, where $p^u$ is in position $t+v$.   The case $u = s+t$ arises similarly by considering $P[b]P[a]$.

Next we consider the initial term $c=0$, with $a \neq b$.  The $c=0$ terms in $P[a]P[b]$ and $P[b]P[a]$ are equal (to $P(0, \ldots, 0, p^u, 0, \ldots, 0, p^s)$ if $t>v$, and to $P(0, \ldots, 0, p^u +p^s)$ if $t=v$, $s>u$), and so they cancel in $[P[a],P[b]]$.

Finally  we consider the initial term $c=0$, with $a=b$, so that $t=v$ and $s=u$.   The $c=0$ term is $2P(0, \ldots, 0, 2p^s)$ since the binomial coefficient $\binom{2p^s}{p^s} = 2$ mod $p$.   This gives the result $P[a]^2 = 0$ in the case $p=2$.   For $p>2$ we need to consider products $P(0, \ldots, 0, ip^s)P(0, \ldots, 0, jp^s)$ for $0 \le i,j \le p-1$.   These produce similar Milnor matrices to those shown above.   Again the case $c>0$ gives terms of higher May filtration.   The initial Milnor matrix $c=0$ gives a multiple $\binom{(i+j)p^s}{ip^s} = \binom{i+j}{i}$, which is nonzero mod $p$ when $i+j <p$ and is $0$ mod $p$ when $i+j=p$.   It follows that $P[a]^p =0$ in $E^0(\A_q)$.
\end{proof}

Proposition \ref{prop_power_comm} allows us to determine the dimension of the $k$th filtration quotient $(\A_q^+)^k/(\A_q^+)^{k+1}$ as a vector space over $\FF_q$ for each degree $d \ge 0$.  Since $E^0(\A_q)$ has a basis given by taking the products of the generators $P[a]$ in a fixed but arbitrary order, this dimension is the number of multisets $A = \{a_1, a_2, \ldots, a_\ell\}$ of $q$-atomic numbers with multiplicities $\le p-1$, sum $|A| = \sum_{i=1}^\ell a_i= d$ and filtration $\sum_{i=1}^\ell M(P[a_i]) = k$.   Thus we have the following result.

\begin{proposition} \label{Pseries_E^0}
The Poincar\'e series of the graded algebra $E^0(\A_p(n-2))$ is
$$
\Pi(E^0(\A_p(n-2)),\tau) = \left(\frac{1- \tau^p}{1-\tau}\right)^{n-1} \left(\frac{1- \tau^{2p}}{1-\tau^2}\right)^{n-2} \cdots \left(\frac{1- \tau^{(n-1)p}}{1-\tau^{n-1}}\right).
$$
For $q=p^e$, the Poincar\'e series of $E^0(\A_q(n-2))$ is the $e^{\it th}$ power of this series.
\qed
\end{proposition}

\begin{example} \label{exa_Pseries} {\rm
The basis for $E^0(\A_2(1))$ given by Proposition \ref{prop_power_comm} is shown below, together with the grading, using the order $1, 3, 2$ on $2$-atomic numbers.   Since $Sq[1] = Sq(1)$, $Sq[2] = Sq(2)$ and $Sq[3] =Sq(0,1)$, this basis is the Milnor basis.   The Poincar\'e series is $(1+\tau)^2(1+\tau^2) =1 + 2\tau + 2\tau^2 +2\tau^3 +\tau^4$.
$$
\begin{array}{|c|c|c|c|c|c|} \hline k&0&1&2&3&4 \\ \hline  &1 & Sq[1] & Sq[3] & Sq[1]Sq[3] & Sq[1]Sq[3]Sq[2] \\ && Sq[2] & Sq[1]Sq[2] & Sq[3]Sq[2] &\\ \hline \end{array}
$$
The corresponding series for $E^0(\A_4(1))$ is $(1+\tau)^4(1+\tau^2)^2 =1 + 4\tau + 8\tau^2 +12\tau^3 +14\tau^4 + 12\tau^5 + 8\tau^6 + 4\tau^7 + \tau^8$.   It can be obtained by applying Proposition \ref{May_filt_Mil_q} to the Milnor basis elements $P(r_1, r_2)$ with $0 \le r_1 \le 15$ and $0 \le r_2 \le 3$.
}
\end{example}

For $M \ge 0$, there is a natural bijection between monomials $P[A]$ in the generators $P[a]$ with exponents $\le p-1$, degree $d$ and grading $M$ in $E^0(\A_q)$, and Milnor basis elements $P(R) = P(r_1. r_2, \ldots)$ with degree $d$ and May filtration $M$ in $\A_q$.   This has been described, by K.~G.~Monks \cite{Monks} for $p=2$ and by D.~Yu.~Emelyanov and Th.~Yu.~Popelensky \cite{Emelyanov-Popelensky17} for $p>2$, as follows: if $P[a] = P^s_t$, so that $a= p^s(q^t-1)/(q-1)$, then $P[a]$ appears with exponent $k$ in $P[A]$ if and only if $p^s$ has coefficient $k$ in the base $p$ expansion of $r_t$.   Hence $d = \sum_j k_ja_j = \sum_i r_i(q^i-1)/(q-1)$, since if $a_j = p^{s_j}(q^{t_j}-1)/(q-1)$ then $r_i = \sum_{t_j=i} k_jp^{s_j}$ and $\alpha(r_i) = \sum_{t_j=i} k_j$, and $M = \sum_j k_j t_j = \sum_i i\, \alpha(r_i)$.

\begin{example} \label{exa_Monks_corr_1} {\rm
The first few $3$-atomic numbers are shown below.   The monomial $P[1]P[4]^2P[3]^2$ corresponds to the Milnor basis element $P(7,2)$, with $d = 15$ and $M = 7$.
$$
\begin{array}{ccccc} \vdots \\27 & \vdots \\9& 36 & \vdots \\ 3 & 12 & 39 & \vdots\\ 1&4&13&40& \cdots
\end{array}
$$
}
\end{example}

\begin{example} \label{exa_Monks_corr_2} {\rm
The first few $4$-atomic numbers are shown below.   The monomial $P[1]P[10]P[8]$ corresponds to the Milnor basis element $P(9,2)$, with $d = 19$ and $M = 4$.
$$
\begin{array}{ccccc} \vdots \\8 & \vdots\\4& 20& \vdots  \\ 2 & 10 & 42 &\vdots  \\ 1&5&21&85&\cdots
\end{array}
$$
}
\end{example}

A sum of elements of filtration $M$ in a filtered algebra can have filtration $>M$, as the next example shows.

\begin{example} \label{exa_deg_12} {\rm
The following table shows the grading of the Milnor basis in $E^0(\A_2)$ for elements of degree $12$.
$$
\begin{array}{|c|c|c|c|c|c|} \hline {\rm grading}  & 2 & 3 & 4 & 5 & 6 \\ \hline {\rm Milnor\ basis} & Sq(0,4)&\qquad& Sq(6,2) &Sq(5,0,1) & Sq(2,1,1) \\ & Sq(12) &\qquad& Sq(9,1) & & Sq(3,3) \\ \hline
\end{array}
$$
Thus for $q=2$ and degree $d=12$, $(\A_2^+)^6/(\A_2^+)^7$ is spanned by the elements
\begin{eqnarray*}
Sq^1Sq^2Sq^4Sq^2Sq^1Sq^2 = Sq(3,3) &=& Sq^9Sq^3,\\
Sq^2Sq^4Sq^2Sq^1Sq^2Sq^1 = Sq(2,1,1)+ Sq(3,3) &=& Sq^8Sq^3Sq^1 + Sq^9Sq^2Sq^1.
\end{eqnarray*}
But $Sq^8Sq^3Sq^1$ and $Sq^9Sq^2Sq^1$ have filtration $4$, since their Milnor basis expansion involves $Sq(9,1)$ in each case.   It follows that there is no `admissible basis' of the filtration quotients of $\A_2$.
}
\end{example}

If a similar situation were to occur for a sum of Milnor basis elements of the same May filtration in $\A_q$, then, in view of the bijection described above, some filtration quotient of $\A_q$ would have dimension greater than that given by Proposition \ref{Pseries_E^0}.  Thus, by working down from the highest filtration in a given degree, we obtain the following result \cite{May}.

\begin{proposition} \label{Mil_gr}
For any element $\theta \in \A_q$, the May filtration of $\theta$ is the minimum of the May filtrations of the terms in the expansion of $\theta$ in the Milnor basis. \qed
\end{proposition}

\section{The graded algebra $E^0(\FF_p \U(n))$} \label{sec_Quillen}

Given a group $G$ and a field $F$, the augmentation ideal $A$ of the group algebra $FG$ is the kernel of the homomorphism $FG \rightarrow F$ which maps $g \mapsto 1$ for all $g \in G$.   We filter $FG$ by the powers $A^i$ of $A$, and define the associated graded algebra $E^0(FG) = \oplus_{i \ge 0} A^i/A^{i+1}$, where $A^0 = FG$ and the product of $a \in A^i$ and $b \in A^j$ is defined by $\overline{a}\,\overline{b} = \overline{ab} \in A^{i+j}$, where $\overline{a}$ and $\overline{b}$ are the cosets $a + A^{i+1}$ and $b + A^{j+1}$.   In this section, we consider $E^0(\FF_p \U(n))$ where $p$ is prime and $\U(n) \subset GL(n, \FF_p)$ is the subgroup of upper unitriangular matrices, and we prove Theorem \ref{th_Priddy} by finding generators and defining relations for $E^0(\FF_p \U(n))$ which correspond to those of Proposition \ref{prop_power_comm} for $E^0(\A_p(n-2))$.

\begin{example} \label{exa_3} {\rm
In the case $p=2$, the group $\U(3)$ is dihedral of order $8$.   It is generated by $e_1, e_2, e_3$ where
$$
e_1 = \begin{pmatrix} 1&1&0\\0&1&0\\0&0&1 \end{pmatrix}, \quad e_2 = \begin{pmatrix} 1&0&0\\0&1&1\\0&0&1 \end{pmatrix}, \quad e_3 = \begin{pmatrix} 1&0&1\\0&1&0\\0&0&1 \end{pmatrix},
$$
with defining relations $e_1^2 =e_2^2 =e_3^2 =I$, $(e_1, e_3) =(e_2, e_3) =I$ and $(e_1,e_2) =e_3$, where $I$ is the identity matrix and $(g_1,g_2) = g_1^{-1}g_2^{-1}g_1g_2$ is the commutator.   The augmentation ideal $A$ of $\F\U(3)$ is given by formal sums of an even number of matrices, and is generated by $f_1 = 1+e_1$, $f_2 =1+e_2$ and $f_3 =1+e_3$, where $1 =I$.   These satisfy the relations $f_1^2 =f_2^2 =f_3^2 =0$, $[f_1, f_3] =[f_2, f_3] =0$ and $[f_1,f_2] = f_3 + f_1f_3 + f_2f_3 + f_1f_2f_3$, where $[a_1, a_2] = a_1a_2 + a_2a_1$  is the commutator.   The last relation shows that $f_3 \in A^2$ and that $[f_1, f_2] =f_3$ in the graded algebra $E^0(\F\U(3))$.   This algebra has dimension $8$ over $\F$, with basis $\{1, f_1, f_2, f_3, f_1f_2, f_1f_3, f_2f_3, f_1f_2f_3\}$ and defining relations $f_1^2 =f_2^2 =f_3^2 =0$, $[f_1, f_3] =[f_2, f_3] =0$ and $[f_1,f_2] = f_3$.   The grading is shown by the diagram
$$
\begin{array}{|c|c|c|c|c|} \hline 0&1&2&3&4 \\ \hline  1 & f_1 & f_3 & f_1f_3 & f_1f_2f_3 \\ & f_2 & f_1f_2 & f_2f_3 &\\ \hline \end{array}\ .
$$
In this case, the corresponding basis for $E^0(\A_2(1))$ obtained using $f_1 \leftrightarrow Sq^1$ and $f_2 \leftrightarrow Sq^2$ is the Milnor basis
$$
\begin{array}{|c|c|c|c|c|c|} \hline 0&1&2&3&4 \\ \hline  1 & Sq(1) & Sq(0,1) & Sq(1,1) & Sq(3,1) \\ & Sq(2) & Sq(3) & Sq(2,1) &\\ \hline \end{array}\ .
$$
}
\end{example}

A theorem of Quillen \cite{Quillen} describes the structure of $E^0(FG)$ as the universal enveloping algebra of a Lie algebra $L(G)$ associated to $G$.   If $F$ has prime characteristic $p$, then $L(G)$ is a `restricted' Lie algebra, as it has an additional structure map called a `$p$th power map'.   The universal enveloping algebra is also taken in the restricted sense.

The construction of the Lie algebra $L(G)$ is based on the dimension series of the group $G$.   The $k$th dimension subgroup $D_k(G)$ is the normal subgroup of $G$ which consists of elements $g \in G$ such that $g-1 \in A^k$.   In particular, $D_1(G) =G$.   The quotients $D_k(G)/D_{k+1}(G)$ are Abelian, and, for the purpose of constructing $L(G)$ as their direct sum, these quotients are written additively.   Thus $L(G) = \oplus_{k \ge 1} D_k(G)/D_{k+1}(G)$ is a graded Abelian group, and we define a Lie product on $L(G)$ by $[\overline{g_1}, \overline{g_2}] = \overline{(g_1,g_2)}$, where $(g_1,g_2) = g_1^{-1}g_2^{-1}g_1g_2$ is the commutator in $G$.   Various identities for commutators in $G$ then translate into bilinearity, anti-commutativity and the Jacobi identity in $L(G)$ \cite{Passi}.

We embed $L(G)$ additively into $E^0(FG)$ by the map $\theta: \overline{g} \mapsto g - 1 + A^{k+1}$, where $g$ has grading $k$ in $L(G)$.   As it is an associative algebra, $E^0(FG)$ also has a Lie product defined by $[\overline{g_1}, \overline{g_2}] = \overline{g_1}\, \overline{g_2} - \overline{g_2}\, \overline{g_1}$, and the map $\theta$ is a homomorphism of Lie rings.   Finally we take coefficients in $F$, to get a map $\theta:L(G) \otimes_{\mathbb Z} F \rightarrow E^0(FG)$ of Lie algebras over $F$.   The $p$th power map on $E^0(FG)$ is defined by ${\overline{g}}^{[p]} = {\overline{g}}^p$.

\begin{example} \label{exa_L3} {\rm
In Example \ref{exa_3}, $D_2 =\{1, e_3\}$ is the centre of $\U(3)$, and $D_3 = \{1\}$.   Regarded as elements of $L(\U(3))$, the cosets $1+D_2$, $e_1+D_2$, $e_2 +D_2$, $e_1e_2 + D_2$ are written as $0$, $\overline{e_1}$, $\overline{e_2}$ and $\overline{e_1} + \overline{e_2}$ respectively, and the cosets $1+D_3$, $e_3+D_3$ as $0$, $\overline{e_3}$.   The commutator relations in Example \ref{exa_3} show that the Lie product in $L(\U(3))$ is given by $[\overline{e_1}, \overline{e_3}] = [\overline{e_2},\overline{e_3}] =0$ and $[\overline{e_1}, \overline{e_2}] = \overline{e_3}$.   The elements $\overline{e_1}$, $\overline{e_2}$ of $L(\U(3))$ have grading $1$, while $\overline{e_3}$ has grading $2$.

The embedding $\theta: L(\U(3)) \rightarrow E^0(\F\U(3))$  maps $\overline{e_i}$ to $f_i$ for $i =1,2,3$.   The relations in $E^0(\F\U(3))$ of Example \ref{exa_3} show that $\theta$ is a map of Lie algebras, and that $\overline{g} \mapsto {\overline{g}}^{[2]}$ is the zero map.   The (ungraded) Lie algebra $L(\U(3))$ is isomorphic to the Lie algebra of nilpotent upper triangular $3 \times 3$ matrices over $\F$ \cite{Mitchell}.
}
\end{example}

Given a restricted Lie algebra $L$ over a field $F$ of characteristic $p$, the restricted universal enveloping algebra $U(L)$ is the quotient of the tensor algebra of $L$ by the relations $[a,b] = a \otimes b - b\otimes a$ and $a^{[p]} = a \otimes a \otimes \cdots \otimes a$ ($p$ factors), for all $a, b \in L$.   Then $U(L)$ is an associative algebra, and has a Lie product defined by commutators.  The natural map $L \rightarrow U(L)$ is then an embedding of restricted Lie algebras.   By the Poincar\'e-Birkhoff-Witt theorem, if $x_1, x_2, \ldots, x_m$ is an ordered $F$-basis for $L$, then a $F$-basis for $U(L)$ is given by the monomials $x_{i_1}^{\alpha_1}x_{i_2}^{\alpha_2} \cdots x_{i_r}^{\alpha_r}$ with $i_1 < i_2 < \cdots < i_r$ and $0 \le \alpha_i < p$.

\begin{theorem} {\rm (Quillen, \cite{Quillen})} \label{th_Quillen}
The map $\theta:L(G) \otimes_{\mathbb Z} F \rightarrow E^0(FG)$ extends to an isomorphism
$$
\theta: U(L(G) \otimes_{\mathbb Z} F) \rightarrow E^0(FG)
$$
of graded algebras over $F$.   If $F$ has prime characteristic $p$, then $U(L(G) \otimes_{\mathbb Z} F)$ is taken in the restricted sense.
\end{theorem}

When $G = \U(n)$, the group of upper unitriangular $n \times n$ matrices over $\FF_p$, the dimension subgroups are the terms of its lower central series: $D_k(\U(n))$ consists of all matrices $U = (u_{i,j})$ with $k-1$ diagonals of zeros above the main diagonal, so that $u_{i,j} = 0$ for $1 \le j-i < k$.   Thus $D_k(\U(n))/D_{k+1}(\U(n))$ is an elementary Abelian group of order $p^{n-k}$.   To construct the Lie algebra $L(\U(n))$, we write the group $D_k(\U(n))/D_{k+1}(\U(n))$ additively, so as to regard it as vector space of dimension $n-k$ over $\FF_p$.

For $1 \le i< j \le n$, let $E_{i,j} \in \U(n)$ be the elementary $n \times n$ matrix over $\FF_p$ obtained  by changing the $(i,j)$th entry of the identity matrix $I$ from $0$ to $1$.
The commutator $(E_{i,j}, E_{k,\ell}) = E_{i,j}^{-1} E_{k,\ell}^{-1}E_{i,j} E_{k,\ell}$ is $E_{i,\ell}$ if $j=k$, $E_{k ,j}$ if $i =\ell$, and $I$ otherwise.

We write $E_{i,j}= e_a$, where $a = (p^{j-1}- p^{i-1})/(p-1)$ is a $p$-atomic number.   The conditions $1 \le i < j \le n$ on $i$ and $j$ then correspond to the conditions $1 \le a \le (p^{n-1}-1)/(p-1)$, and for $p$-atomic numbers $a$ and $b = (p^{k-1}- p^{\ell-1})/(p-1)$ the commutator $(e_a, e_b) = e_c$ where $c = a+b = (p^{\ell-1}-p^{i-1})/(p-1)$ if $j=k$, $(e_a, e_b) = e_{c'}$ where $c' = a+b = (p^{j-1}-p^{k-1})/(p-1)$ if $i=\ell$, and  $(e_a, e_b) = I$ otherwise.   Thus, in terms of the (reverse) base $p$  expansions $0 \cdots 0\ 1 \cdots 1$ of $a$ and $b$, $(e_a, e_b) = e_{a+b}$ if $a>b$ and the blocks of $1$s in the expansions of $a$ and $b$ abut, so as to form a single block of $1$s under base $p$ addition, as shown below,
$$
\begin{array}{ccc} 0 \cdots 0 & 1 \cdots 1 \\ 0 \cdots 0 & 0 \cdots 0 & 1 \cdots 1 \\ \hline  0 \cdots 0 & 1 \cdots 1  & 1 \cdots 1 \end{array}
$$
 and otherwise $e_a$ and $e_b$ commute.  Thus $(e_a, e_b) = e_{a+b}$ if $a>b$ and $a+b$ is $p$-atomic and is not a power of $p$, and otherwise $(e_a, e_b) = I$.

The quotient group $D_k(\U(n))/D_{k+1}(\U(n))$ is  generated by the cosets $\overline{e_a} = e_a + D_{k+1}(\U(n))$ of the elements $e_a$ with an entry $1$ on the $k$th superdiagonal $j-i =k$.   As a vector space over $\FF_p$, $L(\U(n))$ is the direct sum of these quotients, and has dimension $n(n-1)/2$.   Thus $L(\U(n))$ is the graded elementary Abelian $p$-group generated by the elements $\overline{e_a}$ for $a \le (p^{n-1}-1)/(p-1)$, and $\overline{e_a}$ has grading $k = j-i$.  Thus the elements of $L(\U(n))$ are formal sums of elements $\overline{e_a}$ corresponding to the matrices $e_a \in \U(n)$.

The Lie product in $L(\U(n))$ is defined by $[\overline{e_a}, \overline{e_b}] = \overline{(e_a,e_b)}$.   It follows from the above discussion of commutators in $\U(n)$ that $[\overline{e_a}, \overline{e_b}] =  \overline{e_{a+b}}$ if $a>b$, $a+b$ is $p$-atomic and is not a power of $p$, and otherwise $[\overline{e_a}, \overline{e_b}] = 0$.   Since $\overline{e_{a+b}}$ has grading $k+\ell$ if $\overline{e_a}$ has grading $k$ and $\overline{e_b}$ has grading $\ell$, this product makes $L(\U(n))$ into a graded Lie algebra.

For each $e_a \in \U(n)$ as above, we define $f_a = 1-e_a$ in $\FF_p\U(n)$.   If $e_a$ has a $1$ on the $k$th superdiagonal, then $f_a \in A^k$, and so $f_a +A^{k+1}$ is a well-defined element of the graded algebra $E^0(\FF_p\U(n))$.    Since the elements $\overline{e_a}$ form a $\FF_p$-basis for $L(\U(n))$, we can define a  map $\theta:L(\U(n)) \rightarrow E^0(\FF_p\U(n))$ of graded algebras by $\theta(\overline{e_a}) = f_a + A^{k+1}$, where $\overline{e_a}$ has grading $k =j-i$ in $L(\U(n))$ and $a = (p^{j-1}- p^{i-1})/(p-1)$.

We shall show that $\theta$ is an embedding of Lie algebras, where the Lie product of $x,y \in E^0(\FF_p\U(n))$ is defined by $[x,y] = xy-yx$.    Thus we wish to prove
\begin{eqnarray*}
\theta([\overline{e_a}, \overline{e_b}]) &=& \theta(\overline{e_a})\theta(\overline{e_b}) - \theta(\overline{e_b})\theta(\overline{e_a}) \\
&=& (f_a + A^{k+1})(f_b + A^{\ell+1})-  (f_b + A^{\ell+1})(f_a + A^{k+1})
\end{eqnarray*}
where $\overline{e_a}$ has grading $k$ and $\overline{e_b}$ has grading $\ell$.   Since products such as $f_a \cdot A^{\ell+1}$ are in $A^{k+\ell+1}$, this simplifies to
$$
\theta([\overline{e_a}, \overline{e_b}]) =  f_a f_b - f_b f_a +A^{k+\ell+1} =  e_a e_b -  e_b e_a + A^{k+\ell+1}.
$$
Since $[\overline{e_a}, \overline{e_b}] = \overline{e_{a+b}}$ if $a+b$ is $p$-atomic, and $[\overline{e_a}, \overline{e_b}] = 0$ otherwise, $\theta([\overline{e_a}, \overline{e_b}]) = 1- e_{a+b} + A^{k+\ell+1}$ if $a+b$ is $p$-atomic, and $\theta([\overline{e_a}, \overline{e_b}]) = 0 + A^{k+\ell+1}$ otherwise.   The result holds in the second case since $e_a$ and $e_b$ commute in $\U(n)$.   In the first case, $e_a^{-1} e_b^{-1} e_a e_b = e_{a+b}$, so $e_b e_a = e_a e_b e_{a+b}$.   Hence $e_a e_b - e_b e_a = e_a e_b(1 - e_{a+b})$.   Since $1-e_{a+b} \in A^{k+\ell}$ and $e_a e_b \in 1+ A$, $e_a e_b - e_b e_a +A^{k+\ell+1}= 1+e_{a+b} +A^{k+\ell+1}$.   Thus the result holds in the first case.  It follows that $\theta$ is a Lie algebra homomorphism.

We next apply Theorem \ref{th_Quillen}.   The restricted universal enveloping algebra $U(L(\U(n)))$ is an associative algebra over $\FF_p$ of dimension $p^{n(n-1)/2}$.   It has a basis given by products of the elements $\overline{e_a}^k$, where $0 \le k \le p-1$ and $a = (p^{j-1} - p^{i-1})/(p-1)$ for $1 \le i < j \le n$, the $p$-atomic numbers $a$ being taken in a fixed but arbitrary order.   The $p$th power map $x \mapsto x^{[p]}$ is trivial since $(e_a+1)^p =0$ in $\FF_p(\U(n))$.   By Quillen's theorem, the map $\theta: L(\U(n)) \rightarrow E^0(\FF_p\U(n))$ extends to an isomorphism $\theta: U(L(\U(n))) \rightarrow E^0(\FF_p\U(n))$ of graded algebras over $\FF_p$.   On basis elements we have
$$
\theta(\overline{e_{a_1}} \otimes \overline{e_{a_2}} \otimes \cdots \otimes \overline{e_{a_r}}) = f_{a_1} f_{a_2}  \cdots f_{a_r} + A^{k+1}
$$
where $a_1, a_2, \ldots, a_r$ are taken in the preferred order,  and $k$ is the sum of the gradings of the elements $\overline{e_{a_i}}$.

In the associative algebra $U(L(\U(n)))$, the Lie bracket $[x,y] =xy - yx$, and so the relations $\overline{e_a}^p =0$ and $\overline{e_a} \, \overline{e_b} - \overline{e_a} \, \overline{e_b} = [\overline{e_a}, \overline{e_b}] = \overline{e_{a+b}}$ or $0$ can be regarded as power-commutator relations defining $U(L(\U(n)))$ as an algebra.   Thus Quillen's theorem provides us with a corresponding definition of the algebra $E^0(\FF_p\U(n))$ by generators and relations.   We write $f_a = 1- e_a \in \FF_p(\U(n))$ where $a$ is $p$-atomic, and translate these relations as $f_a^p =0$ and $f_af_b - f_bf_a = f_{a+b}$ if $a>b$ and $a+b$ is $p$-atomic, and $f_af_b = f_bf_a$ otherwise.

\begin{proof}[Proof of Theorem \ref{th_Priddy}]
The graded algebras $E^0(\FF_p\U(n))$ and $E^0(\A_p(n-2))$ have the same dimension $p^{n(n-1)/2}$ as vector spaces over $\FF_p$.   The discussion above shows that the elements $f_a$ corresponding to the $n(n-1)/2$ $p$-atomic numbers $a$ such that $1 \le a \le (p^{n-1}-1)/(p-1)$ generate $E^0(\FF_p\U(n))$, and satisfy relations which correspond to those of Proposition \ref{prop_power_comm} for the generators  $P^s_t$ of $E^0(\A_p(n-2))$.   Hence the map $E^0(\FF_p\U(n)) \longrightarrow E^0(\A_p(n-2))$ which sends $f_a$ to $P^s_t$, where $a = p^s(p^t-1)/(p-1)$, is an isomorphism of graded algebras.

Further, $E^0(\FF_p\U(n))$ and $E^0(\A_p(n-2))$ are isomorphic as Hopf algebras.  The generators $P[a]$ of $E^0(\A_p(n-2))$ and $f_a$ of $E^0(\FF_p\U(n))$ are coproduct primitives.   Since $P[a] = P^s_t = P(0, \ldots, 0, p^s)$,
$
\phi(P(0, \ldots, 0, k)) = \sum_{i+j=k} P(0, \ldots, 0, i) \otimes P(0, \ldots, 0, j).
$
With the nonzero elements in position $t$, the May filtration of $P(0, \ldots, 0, k)$ is $t\alpha(k)$, and that of $P(0, \ldots, 0, i) \otimes P(0, \ldots, 0, j)$ is $t(\alpha(i) + \alpha(j))$, so in $E^0(\A_q)$ the sum is over $i, j$ such that $i+j=k$ and $\alpha(i) + \alpha(j) =\alpha(k)$.   Since $\alpha(k) =1$ when $k=p^s$, the only terms which survive are given by $i=0$, $j=k$ and $i=k$, $j=0$.

The coproduct in $\FF_p\U(n)$ is given by $\phi(g) = g \otimes g$ for all $g \in \U(n)$.   Hence $\phi(f_a) = \phi(1-e_a) = \phi(1) -\phi(e_a) = 1 \otimes 1 - e_a \otimes e_a = 1 \otimes 1 - (1 -f_a) \otimes (1 -f_a) = f_a \otimes 1 +1 \otimes f_a - f_a \otimes f_a$ for the coproduct in $\FF_p\U(n)$.   Hence in $E^0(\FF_p\U(n))$ we have $\phi(f_a) =   f_a \otimes 1 +1 \otimes f_a$.  It follows that the isomorphism $f_a \leftrightarrow P[a]$ preserves coproducts.   The antipodes are also preserved: since $P[a]$ and $f_a$ are primitive, $\chi(P[a]) = -P[a]$ and $\chi(f_a) = -f_a$.
\end{proof}

\begin{remark} {\rm
By Definition \ref{def_P_e(R)}, for $q =p^e$ the subalgebra $\A'_q(n-2)$ of $\A'_q$ is contained in the subalgebra $\A_p(ne-2)$ of $\A_p$.   The isomorphism $E^0(\FF_p\U(n)) \longrightarrow E^0(\A_p(n-2))$ of Theorem \ref{th_Priddy} maps $E^0(\A'_q(n-2))$ to the subalgebra of $E^0(\FF_p\U(n))$ generated by matrices $U = (u_{i,j})$ in $\U(n)$ with non-zero entries only on every $e$th superdiagonal, so that $u_{i,j} =0$ if $e$ does not divide $j-i$.   Using Proposition \ref{Aq_iso}, we have a corresponding description of $E^0(\A_q(n-2))$.
}
\end{remark}

\section{$P^s_t$ bases for $E^0(\A_q)$} \label{sec_Pst}

Let $p$ be a prime and let $q = p^e$ be a power of $p$.   Following Wood \cite{Wood95}, we define the Y- and Z- orders on $p$- and $q$-atomic numbers as follows.

\begin{definition} \label{def_YZ_order} {\rm
The Y-order on $p$-atomic numbers $p^s(p^t-1)/(p-1)$ is the left lexicographic order on pairs $(s,t)$, while the Z-order is the left lexicographic order on pairs $(s+t,s)$, where $s \ge 0$ and $t \ge 1$.   More generally, the Y- and Z- orders on $q$-atomic numbers $p^s(q^t-1)/(q-1)$ are the left lexicographic orders on the triples $(s', t, s'')$ and $(s'+t, s', s'')$ respectively, where $s =s'e+s''$, $0 \le s'' < e$.
}
\end{definition}

\begin{example} \label{exa_YZ_1} {\rm
Example \ref{exa_Monks_corr_1} shows the first few $3$-atomic numbers.   The Y-order $1, 4, 13, 40, \ldots, 3, 12, 39, \ldots, 9, 36, \ldots, 27, \ldots$ takes rows left to right and upwards.   The Z-order $1, 4, 3, 13, 12, 9, 40, 39, 36, 27, \ldots$ takes diagonals right to left and upwards.
}
\end{example}

\begin{example} \label{exa_YZ_2} {\rm
Example \ref{exa_Monks_corr_2} shows the first few $4$-atomic numbers.   Since $e=2$, we apply the previous recipes to rows taken in pairs.   The Y-order is thus $1, 2, 5, 10, 21, 42, 85, \ldots, 4, 8, 20, 40, \ldots, 16, 32, \ldots$, while the Z-order is $1, 2, 5, 10,$ $4, 8, 21, 42, 20, 40, 16, 32, 85,\ldots$.
}
\end{example}

The $P^s_t$ bases of $\A_p$ defined by Monks \cite{Monks} and Emelyanov-Popelensky \cite{Emelyanov-Popelensky15} are obtained by fixing an arbitrary order on the set of $p$-atomic numbers,  and then taking products of the elements  $P^s_t$ in weakly increasing order, with each such element being repeated up to $p-1$ times.   The bijection described in Section \ref{sec_graded alg} gives a triangular relation between these products and the Milnor basis.    This relationship is sharpened in Proposition \ref{prop_YZ_Pst}.For the Y- and Z-$P^s_t$ bases, which are defined using the Y- and Z- orders).

Given such a basis for $\A_p$, the basis elements which are products of elements $P^s_t$ such that $e$ divides $t$ give a basis for the subalgebra $\A'_q$, where $q =p^e$.   The corresponding elements of $\A_q$ give a  $P^s_t$-basis for $\A_q$, whose elements are products of the elements $P[a] = P^s_t = P(0, \ldots, 0, p^s)$, with one such element in each $q$-atomic degree $a = p^s(q^t-1)/(q-1)$, each being repeated up to $p-1$ times.

\begin{example} {\rm
Let $p=2$, so that $P^s_t = Sq(0, \ldots, 0, 2^s)$ has degree $a=2^s(2^t-1)$.   We denote this element alternatively by $Sq[a]$, and use the Z-order.   The table below augments that of Example \ref{exa_deg_12} by showing the Z-$P^s_t$ elements corresponding to the Milnor basis, abbreviating a product $Sq[a_1]Sq[a_2] \cdots Sq[a_r]$ as $Sq[a_1, a_2, \ldots, a_r]$, with the $a_i$ in increasing Z-order $1,3,2,7,6,4,15,14,12,8, \ldots$.

$$
\begin{array}{|c|c|c|c|c|} \hline \I^2/\I^3 & \I^3/\I^4 & \I^4/\I^5 & \I^5/\I^6 & \I^6/\I^7 \\ \hline Sq(0,4)&& Sq(6,2) &Sq(5,0,1) & Sq(2,1,1) \\ Sq(12) && Sq(9,1) & & Sq(3,3) \\ \hline
Sq[12] && Sq[2,6,4] & Sq[1,7,4] & Sq[3,2,7] \\ Sq[4,8] && Sq[1,3,8] && Sq[1,3,2,6] \\ \hline
\end{array}
$$
}
\end{example}

In this example, the Y-$P^s_t$ basis is obtained from the Z-$P^s_t$ basis by replacing $Sq[2]Sq[7]$ by $Sq[7]Sq[2]$.   Since $Sq(2)Sq(0,0,1) = Sq(2,0,1) = Sq(0,0,1)Sq(2)$, these bases coincide in degree $12$.   (They differ in degree $18$, as $Sq(4)Sq(0,0,2) \neq Sq(0,0,2)Sq(4)$, but the `error term' $Sq(3,0,0,1)$ has higher May filtration.)    The transition matrix from either $P^s_t$ basis to the Milnor basis in degree $12$ is shown below.   Note that the diagonal submatrices corresponding to filtrations $2,4,5$ and $6$ are identity matrices.   We show that this situation holds more generally.

{\footnotesize
$$
\begin{array}{|c|cc|cc|c|cc|} \hline & Sq(0,4) & Sq(12) & Sq(6,2) & Sq(9,1) & Sq(5,0,1)& Sq(2,1,1) & Sq(3,3)   \\ &2&2&4&4&5&6&6 \\ \hline Sq[12] &1& 0&0&0& 0&0& 0 \\ Sq[4,8]  && 1&1&0& 0&0&0\\ \hline Sq[2,6,4] &&& 1& 0&0&1& 1 \\ Sq[1,3,8]  &&&& 1&0&0& 0\\ \hline Sq[1,7,4] &&&&& 1 &0& 0\\ \hline Sq[3,2,7] &&&&&& 1& 0 \\ Sq[1,3,2,6] &&&&&&&  1\\ \hline \end{array}
$$
}

\begin{proposition} \label{prop_YZ_Pst}
The Y- and Z-$P^s_t$ bases of $\A_q$ coincide with the Milnor basis up to elements of higher May filtration, so that all three bases give the same basis of $E^0(\A_q)$.
\end{proposition}

\begin{proof}
Let $a$ and $b$ be $q$-atomic numbers.   In the cases where $P[a]$ and $P[b]$ do not commute in $E^0(\A_q)$, $a$ and $b$ occur in the same order in the Y- and Z-orderings.   Hence the Y- and Z-$P^s_t$ bases of $\A_q$ coincide up to elements of higher May filtration.

To prove that the Milnor basis of $E^0(\A_q)$ also coincides with these, we consider a Z-basis element $P[a_1]P[a_2] \cdots P[a_k]$, where $a_1, a_2, \ldots, a_k$ is a weakly increasing sequence of $q$-atomic numbers with $a_{i+p-1} >_Z a_i$ for all $i$.   We assume by induction on $k$ that $P[a_2] \cdots P[a_k] = P(R)$ in $E^0(\A_q)$, where $R$ is the sequence corresponding to $a_2, \ldots, a_k$ as in Section \ref{sec_graded alg} (see Examples \ref{exa_Monks_corr_1} and \ref{exa_Monks_corr_2}).   For the inductive step, let $a = p^s(q^t-1)/(q-1)$.   We use the Milnor product formula to show that $P[a]P(R) =  P(R')$ in $E^0(\A_q)$ where $R' = R + (0, \ldots, 0, p^s)$, with $p^s$ in position $t$.   We have to consider Milnor matrices
$$
\begin{array}{c|ccccc} &r_1&r_2& \cdots & r_s & \cdots \\ \hline 0 & 0& 0& \cdots & 0 \\ \vdots & \vdots & \vdots && \vdots \\ 0  & 0& 0 & \cdots & 0 \\ c_0  & c_1 & c_2 & \cdots & c_s & \cdots \end{array}
$$
where $\sum_i p^ic_i = p^s$.

We argue as in the proof of Proposition \ref{May_filt_Mil_1}.   By considering row $t$, all such matrices except the initial matrix $c_0 = p^s$, $c_i =0$ for $i>0$ and the final matrix $c_s =1$, $c_i =0$ for $i \neq s$ give elements of higher May filtration.   By considering column $s$, the same is true for the final matrix.
\end{proof}

\section{The Arnon A basis of $\A_p$} \label{sec_Ar_A_p}

Monks \cite{Monks} has compared a number of bases for $\A_2$ with the Milnor basis \cite{Milnor}, and has shown that several of them are triangularly related to the Milnor basis for a suitable bijection between the two bases and for suitable orderings on them.   His results have been generalized to $\A_p$ by Emelyanov and Popelensky \cite{Emelyanov-Popelensky17}.

For example, the admissible basis of $\A_p$ is triangularly related to the Milnor basis as follows.   Recall that admissible monomials in $\A_p$ are elements $P^A = P^{a_1}P^{a_2} \cdots P^{a_\ell}$, where $a_i \ge pa_{i+1}$ for $1 \le i < \ell$, and that they form a $\FF_p$-basis for $\A_p$.   A triangular relation between this basis and the Milnor basis is obtained by associating to $A$ the sequence $R =  (r_1, r_2, \ldots, r_\ell)$, where $r_i = a_i - pa_{i+1}$, and using the right lexicographic order on the sequences $A$ and $R$.   The table below shows the result in degree $9$ when $p=2$.
$$
\begin{array}{|l|ccccc|}\hline
& Sq(2,0,1) & Sq(0,3) & Sq(3,2) & Sq(6,1) & Sq(9)\\ \hline
Sq^6 Sq^2 Sq^1 & 1&1&1&0&0\\  Sq^6 Sq^3 & 0&1&1&1&0\\  Sq^7 Sq^2 & 0&0&1&0&0\\Sq^8 Sq^1  & 0&0&0&1&1\\ Sq^9 & 0&0&0&0&1\\ \hline
\end{array}
$$
When comparing two bases in this way, we need only specify the bijection and the ordering on one of them, as the ordering on the other is defined by the bijection.   Alternatively, we could use orderings on both bases to define the bijection.

The Arnon A basis of $\A_p$ \cite{Arnon, Emelyanov-Popelensky15} is defined as follows.   We begin with the admissible monomials $X^n_k = P^{p^n}P^{p^{n-1}}\cdots P^{p^k}$, $n \ge k \ge 0$, one in each $p$-atomic degree $p^k(p^{n-k+1}-1)/(p-1)$.   A general Arnon A basis element is a product of these elements taken in Z-order, with each factor $X^n_k$ repeated no more than $p-1$ times.   Thus the Arnon A basis is constructed in the same way as the Z-$P^s_t$ basis, but using the elements $X^n_k$ instead of the elements $P^s_t$.   The basis elements have the form $X^{n_1}_{k_1} \cdot X^{n_2}_{k_2} \cdots X^{n_\ell}_{k_\ell}$, where $(n_i, k_i) <_l (n_{i+p-1}, k_{i+p-1})$ for all $i$.   In particular, the factors are distinct when $p=2$.

In the case $p=2$, the bijection defined by Monks \cite{Monks} between the Arnon A basis and the Milnor basis can be described as follows.    Given a Milnor basis element $Sq(R) = Sq(r_1, r_2, \ldots)$, the corresponding Arnon A basis element has one factor $X^n_k$ corresponding to each element in the binary decompositions of the terms of the sequence $R$.   For each $i$ and $j$ such that $2^i \in \bin(r_j)$ we take $X^n_k$, where $k=i$ and $n = i+j-1$, and multiply these elements in Z-order to form the required Arnon A basis element.  Conversely, given an Arnon A basis element $X^{n_1}_{k_1} \cdot X^{n_2}_{k_2} \cdots X^{n_\ell}_{k_\ell}$, let $r_j$ be the sum of the $2$-powers $2^i$ where $j =n_t-k_t+1$ is the length of one of the elements $X^{n_t}_{k_t}$ and $i=k_t$, and form the Milnor basis element $Sq(R) = Sq(r_1, r_2, \ldots)$.   For a general prime $p$, $\bin(r)$ is replaced by the multiset $\pin(r)$ (Definition \ref{def_pin}).   For example if $p=3$ then the Milnor basis element $P(7,2)$ corresponds to the Arnon A basis element $P^1 \cdot P^3P^1 \cdot P^3P^1 \cdot P^3 \cdot P^3$, since $7 =1+3+3$ and $2 =1+1$ give the weakly Z-increasing sequence $(1,4,4,3,3)$.

Monks (for $p=2$) and Emelyanov-Popelensky (for $p>2$) show that the Arnon A basis and the Milnor basis are triangularly related using this bijection between the two bases and the following ordering on the Arnon A basis.   Each Arnon A basis element is defined by a non-decreasing sequence of $p$-atomic numbers.  These sequences are placed in decreasing left lexicographic order $<_Z$, taken with respect to the Z-order and not the usual increasing order on $p$-atomic numbers.   For example, for $p=2$ and degree $d=9$, the change of basis matrix is as follows.
$$
\begin{array}{|c|ccccc|} \hline &Sq(2,0,1) & Sq(0,3) & Sq(6,1) & Sq(9) & Sq(3,2)  \\ \hline  Sq^2 \cdot Sq^4Sq^2Sq^1  &1&1&0&0&1 \\ Sq^2Sq^1 \cdot Sq^4Sq^2 & &1&1&0&1\\Sq^2Sq^1 \cdot Sq^2 \cdot Sq^4  &&& 1&0&0 \\Sq^1 \cdot Sq^8  &&&& 1&0\\Sq^1 \cdot Sq^2 \cdot Sq^4Sq^2   &&&&& 1 \\ \hline \end{array}
$$
Here the $2$-atomic sequences giving the Arnon A elements are ordered as follows:
$$
(2,7) >_Z (3,6) >_Z (3,2,4) >_Z (1,8) >_Z (1,2,6).
$$

Our main result, Theorem \ref{th_Arnon_r}, states that the Arnon A basis and the Milnor basis are also triangularly related using the same bijection, but with a different choice of orderings.   The ordering on the Arnon A basis is the reversed right lexicographic order $<_r$, where these elements are treated simply as monomials in the generators $P^{p^j}$, $j \ge 0$ of $\A_p$.   This ordering ignores the factorization of Arnon A basis elements as products of $p$-atomic basis elements.   For example, for $p=2$ and $d=9$ the change of basis matrix is as follows.
$$
\begin{array}{|c|ccccc|} \hline & Sq(9) & Sq(2,0,1) & Sq(0,3) & Sq(3,2) & Sq(6,1)  \\ \hline Sq^1 \cdot Sq^8 &1&0&0&0&0 \\ Sq^2 \cdot Sq^4Sq^2Sq^1 & &1&1&1&0\\ Sq^2Sq^1 \cdot Sq^4Sq^2 &&& 1&1&1 \\ Sq^1 \cdot Sq^2 \cdot Sq^4Sq^2  &&&& 1&0\\ Sq^2Sq^1 \cdot Sq^2 \cdot Sq^4 &&&&& 1 \\ \hline \end{array}
$$

\begin{remark} {\rm
The ordering $\le_R$ which appears in \cite{Arnon, Emelyanov-Popelensky15} is not the same as $\le_r$: it is obtained by first right-justifying the exponent sequences and then taking lexicographic order from the right.   For example, $(1,1,2) >_R (3,1) =(0,3,1)$, since $2>1$.   However, $(3,1) = (3,1,0) >_r (1,1,2)$ since $0<2$.
}
\end{remark}

In order to establish the A basis of $\A_2$, Arnon \cite{Arnon} considers the set of all formal monomials in the elements $Sq^{2^j}$, $j \ge 0$, by which we mean elements of the free algebra $\A$ generated by these symbols.   Thus $\A_2$ is the quotient algebra of $\A$ by the two-sided ideal generated by the Adem relations and $Sq^0=1$.   As for $\A_2$, $\A$ is graded by giving $Sq^{2^j}$ degree $2^j$.   When the formal monomials of a given degree are taken in increasing left order $\le_l$, Arnon shows that the minimal monomials form the A basis \cite[Theorem 5(A)]{Arnon}.

If the left lexicographic order $\le_l$ is replaced by $\le_r$, then the minimal monomials do not in general give the A basis.   For $p=2$, the first counterexample is in degree $9$, where the $\le_r$-minimal basis is obtained from the A basis by replacing $Sq^2 \cdot Sq^4Sq^2Sq^1$ with $Sq^4Sq^2Sq^1Sq^2$, which is lower in the (reversed) right order $\le_r$.   The relation
\begin{equation} \label{eqn_min_ex}
Sq^2 \cdot Sq^4Sq^2Sq^1 = Sq^4Sq^2Sq^1 \cdot Sq^2 + Sq^1 \cdot Sq^2 \cdot Sq^4Sq^2 + Sq^2Sq^1 \cdot Sq^2 \cdot Sq^4
\end{equation}
shows that $Sq^2 \cdot Sq^4Sq^2Sq^1$ is reducible in the right order.

We define a variant of the A basis by replacing the Z-order on $2$-atomic degrees by the Y-order, and we refer to the original Arnon A basis as the {\bf Z-Arnon A basis}, the new basis as the {\bf Y-Arnon A basis}.    The Y-Arnon A basis elements have the form $X^{n_1}_{k_1} \cdot X^{n_2}_{k_2} \cdots X^{n_\ell}_{k_\ell}$, where $(n_1,k_1)  <_r  (n_2, k_2) <_r \cdots <_r (n_\ell, k_\ell)$.   Thus the elementary Arnon A monomials $X^n_k$ are multiplied  in Y-order of their degrees.   The Y- and Z-Arnon A bases first differ in degree $(p+1)^2$, by replacing $P^p \cdot P^{p^2}P^pP^1$ by $P^{p^2}P^pP^1 \cdot P^p$.   For $p=2$ this is illustrated by (\ref{eqn_min_ex}).

\section{The Arnon A basis of $\A_q$} \label{sec_Ar_A_q}

In this section we generalize the Y- and Z-Arnon A bases to $\A_q$ and prove our main result, Theorem \ref{th_Arnon_r}.   For $q =p^e$, the elementary Arnon A monomial $X^n_k = P^{p^n}P^{p^{n-e}}\cdots P^{p^k}$ is defined when $n \ge k \ge 0$ and $n = k$ mod $e$, and its degree is the $q$-atomic number $p^k(q^{t+1}-1)/(q-1)$, where $n=k+te$.   A general Y- or Z-Arnon A basis element is a product of these elements taken in Y- or Z-order, with each factor $X^n_k$ repeated no more than $p-1$ times.   In other words, the Y- or Z-Arnon A basis is constructed in the same way as the corresponding $P^s_t$ basis, with the elements $P^s_t$ replaced by the elements $X^n_k$.   Thus the basis elements have the form $X^{n_1}_{k_1} \cdot X^{n_2}_{k_2} \cdots X^{n_\ell}_{k_\ell},\quad (n_1,k_1)  \le_l  (n_2, k_2) \le_l \cdots \le_l (n_\ell, k_\ell)$, where $(n_i, k_i) <_l (n_{i+p-1}, k_{i+p-1})$ for all $i$.

\begin{theorem} \label{th_Arnon_r}
{\rm (i)} The Y-Arnon A monomials in $\A_q$ are the minimal Steenrod monomials in increasing $\le_r$ order.

{\rm (ii)} The Y-Arnon A monomials form a basis of $\A_q$.

{\rm (iii)} The Y- and Z-Arnon A bases of $\A_q$ are triangular with respect to the Milnor basis, using the $\le_r$ order on the Arnon A basis and the bijection with the Milnor basis given by the $q$-atomic sequences.
\end{theorem}

Statement (ii) follows at once from (i).   It follows from (iii) that the Z-Arnon A basis is a basis for $\A_q$, and that it is triangular with respect to the Y-Arnon A basis when both are in the $\le_r$ order.

The proof of Theorem \ref{th_Arnon_r} will occupy the rest of this section.   We begin by expressing the particular elements $X^n_0$ of the Arnon A basis, regarded as elements of the graded algebra $E^0(\A_q)$, in the Z-$P^s_t$  basis.   A typical Z-$P^s_t$ basis element has the form $P[a_1]P[a_2] \cdots P[a_r]$, where $(a_1, a_2, \ldots, a_r)$ is a sequence of $q$-atomic numbers in increasing Z-order, and has May filtration $\sum_{i=1}^r \alpha(a_i)$.

\begin{proposition} \label{prop_Xn0}
Let $q =p^e$ where $p$ is prime, and let $n =te + k$, where $0 \le k < e$.   Then in $E^0(\A_q)$
$$
X^n_k = \sum_{r=k}^t X^{(r-1)e+k}_k P[p^k(q^{t+1}-q^r)/(q-1)], \hbox{\ \rm where\ } X^{-e}_k =1.
$$
\end{proposition}

When $k=0$, $t=1$ this states that $P^qP^1 = P^1P[q] +P[q+1]$, and when $k=0$, $t=2$ that $P^{q^2}P^qP^1 = P^qP^1P[q^2] + P^1P[q^2+q] + P[q^2+q+1]$.   Recall that $P[q] = P^q$, $P[q^2] =P^{q^2}$, $P[q+1] = P(0,1)$, $P[q^2+q] = P(0,q)$ and $P[q^2+q+1] = P(0,0,1)$.

\begin{proof}
We first establish the case $k=0$, namely
$$
X^n_0 = \sum_{r=0}^t X^{(r-1)e}_0 P[(q^{t+1}-q^r)/(q-1)],
$$
by induction on $t$.   Thus let $t \ge 2$, and assume that the above formula holds for $t-1$ in place of $t$.   Since $X^{te}_0 = P^{q^t}X^{(t-1)e}_0$, the induction hypothesis gives $X^{te}_0 = P^{q^t}\sum_{r=0}^{t-1} X^{(r-1)e}_0 P[(q^t -q^r)/(q-1)]$.   Since $P^{q^t}$ and $P^{q^r}$ commute in $E^0(\A_q)$ for $0 \le r \le t-2$, this gives $X^n_0 = \sum_{r=0}^{t-1} X^{(r-1)e}_0 P^{q^t}P[(q^t -q^r)/(q-1)]$.   Using the relation $[P^{q^t},P[(q^t - q^r)/(q-1)]] = P[(q^{t+1} -q^r)/(q-1)]$, it follows that
$$
X^{te}_0 =  \sum_{r=0}^{t-1} X^{(r-1)e}_0 (P[(q^t -q^r)/(q-1)]P^{q^t} + P[(q^{t+1} - q^r)/(q-1)]).
$$
Applying the induction hypothesis again, this gives
\begin{eqnarray*}
X^{te}_0 &=& X^{(t-1)e}_0 P^{q^t} +  \sum_{r=0}^{t-1} X^{(r-1)e}_0 P[(q^{t+1} -q^r)/(q-1)] \\ &=& \sum_{r=0}^t X^{(r-1)e}_0 P[(q^{t+1} - q^r)/(q-1)].
\end{eqnarray*}
This completes the induction, and the proof for $k=0$.   The general case follows, as the relations in $E^0(\A_q)$ are unchanged when all exponents are multiplied by $p$.
\end{proof}

As in \cite{Arnon}, to prove parts (i) and (ii) of Theorem \ref{th_Arnon_r} it suffices to show that any monomial in $\A_q$ which is not of the required form is reducible in the $\le_r$ order, and that the number of monomials of the required form in each degree $d$ is the dimension of $\A_q^d$.   The second statement is clear since, as for the $P^s_t$ bases, there is one elementary Arnon A monomial in each $q$-atomic degree, and these are used to form monomials with exponents $<p$.   For the first statement, we observe that a formal monomial $P^A = P^{p^{j_1}} \cdots P^{p^{j_r}}$  is not a Y-Arnon A basis element if and only if at least one of the following cases occurs:

(1) for some $k$, $j_k > j_{k+1}$ and $j_k -j_{k+1} \neq e$;

(2) the sequence $P[p^m]X^n_k$ with $k<m'e$, $m \le n$ and $m = m'e+m''$, $0 \le m'' <e$, appears in $P^A$;

(3) the sequence $X^n_kP[p^k]$ with $k<n$ appears in $P^A$;

(4) the sequence $X^n_k \cdots X^n_k$ ($p$ factors) appears in $P^A$.

To prove these statements, it is sufficient to work in $E^0(\A_q)$.    Case (1) is resolved immediately, since $P[p^s]$ and $P[p^t]$ commute in  $E^0(\A_q)$ if $t \neq s+e$.

We deal next with Case (2).   For $p=2$, the minimal case $Sq^2 \cdot Sq^4Sq^2Sq^1$ is lowered in the $\le_r$ order by relation (\ref{eqn_min_ex}) above.   We use Proposition \ref{prop_Xn0} to obtain a similar reduction in general.

\begin{proposition} \label{prop_SX}
In $E^0(\A_q)$, $P^{p^m}X^n_k$ can be reduced in the $\le_r$ order for $k<m'e$, $m \le n$ and $m = m'e+m''$, $0 \le m'' <e$.
\end{proposition}

\begin{proof}

We can immediately reduce to $k =m-1$, since for $k<m-1$ we have $P^{p^m}X^n_k = P^{p^m}X^n_{m-1} \cdot X^{m-2}_k$, and a reduction of $P^{p^m}X^n_{m-1}$ in the right order gives a reduction of $P^{p^m}X^n_{m-1} \cdot X^{m-2}_k$ in the right order.

Similarly, we can reduce to $n=m$ or $n=m+e$ using Case (1).   For example, suppose that we wish to reduce $P^p \cdot P^{p^3}P^{p^2}P^pP^1$ in the $\le_r$ order.   Since $P^pP^{p^3} = P^{p^3}P^p$ in $E^0(\A_p)$, it is sufficient to lower $P^p \cdot P^{p^2}P^pP^1$ in the $\le_r$ order.

Thus let $n=m$, $k=m-e$.   We have $P^{p^m} \cdot P^{p^m}P^{p^{m-e}} = P^{p^m}( P^{p^{m-e}}P^{p^m} + P[p^m + p^{m-e}])$ in $E^0(\A_q)$.   The first term is $<_r P^{p^m} \cdot P^{p^m}P^{p^{m-e}}$.   The second term reduces to $P[p^m + p^{m-e}] P^{p^m}$ in $E^0(\A_q)$.   On expanding $P[p^m + p^{m-e}]$ as $P^{p^m}P^{p^{m-e}} -  P^{p^{m-e}}P^{p^m}$, we obtain two terms, and both are $<_r P^{p^m} \cdot P^{p^m}P^{p^{m-e}}$.

Finally let $n=m+e$, $k =m-e$.   By Proposition \ref{prop_Xn0}, we have $P^{p^m} X^{m+e}_{m-e} = P^{p^m}(P^{p^m}P^{p^{m-e}}P[p^{m+e}] + P^{p^{m-e}} P[p^{m+e}+p^m] + P[p^{m+e}+p^m + p^{m-e}])$.   The first term can be lowered in the $\le_r$ order by the previous case, and the second by expanding $P[p^{m+e}+p^m]$ as $P^{p^{m+e}}P^{p^m} -  P^{p^m}P^{p^{m+e}}$.   The third term $P^{p^m} P[p^{m+e}+p^m + p^{m-e}] = P[p^{m+e}+p^m + p^{m-e}] P^{p^m}$ in $E^0(\A_q)$.   The expansion $P[p^{m+e}+p^m + p^{m-e}] = [P^{p^{m+e}}, [P^{p^m}, P^{p^{m-e}}]]$ gives four terms, which are all $<_r  P^{p^m} X^{m+e}_{m-e} = P^{p^m} \cdot P^{p^{m+e}} P^{p^m} P^{p^{m-e}}$.
\end{proof}

We next treat case (3) by the same method.

\begin{proposition} \label{prop_XS}
In $E^0(\A_q)$, $X^n_k P^{p^k}$ can be reduced in the $\le_r$ order for $k<n$, $n=k+te$.
\end{proposition}

\begin{proof}
We can immediately reduce to the case $n=k+e$, since for $n>k+e$ we have $X^n_k P^{p^k} = X^n_{k+2e}X^{k+e}_k P^{p^k}$, and a reduction of $X^{k+e}_k P^{p^k}$ in the right order gives a reduction of $X^n_{k+2e}X^{k+e}_k P^{p^k}$ in the right order.

For $n=k+e$, we have $X^{k+e}_k P^{p^k} = P^{p^{k+e}}P^{p^k}P^{p^k} = (P^{p^k}P^{p^{k+e}} + P[p^{k+e}+ p^k])P^{p^k}$ in $E^0(\A_q)$.   The first term is $<_r P^{p^{k+e}}P^{p^k}P^{p^k}$, and the second term $P[p^{k+e}+ p^k]P^{p^k} = P^{p^k}P[p^{k+e}+ p^k] =  P^{p^k}[P^{p^{k+e}}, P^{p^k}] $ in $E^0(\A_q)$, and both terms are $<_r P^{p^{k+e}}P^{p^k}P^{p^k}$.
\end{proof}

The following example illustrates the proof of case (4), using the same method.

\begin{example} \label{X21_p3} {\rm
Let $p=3$ and consider $X^2_1X^2_1X^2_1 = P^3P^1P^3P^1P^3P^1$.   We reduce this in the right order in $E^0(\A_3)$ as follows.   Recalling that $P[4] = P(0,1)$, we have $P^3P^1P^3P^1P^3P^1 = P^3P^1P^3P^1(P^1P^3 + P[4])$.   The first term is $<_r P^3P^1P^3P^1P^3P^1$, and the second is equal to $P[4] P^3P^1P^3P^1 = P[4] P^3P^1 (P^1P^3 + P[4])$.   Since $P[4] = [P^3,P^1]$, the first term is again $<_r P^3P^1P^3P^1P^3P^1$.   The second term $P[4]P^3P^1P[4] = P[4]P[4]P^3P^1 = P[4]P[4](P^1P^3 + P[4])$.   The first term is again $<_r P^3P^1P^3P^1P^3P^1$, and the second $P[4]P[4]P[4]=0$ in $E^0(\A_3)$.
}
\end{example}

\begin{proposition} \label{prop_XX}
In $E^0(\A_q)$, $X^n_k \cdots X^n_k$ ($p$ factors) can be reduced in the $\le_r$ order for $k \le n$, $n=k+te$.
\end{proposition}

\begin{proof}
We first expand the last factor $X^n_k$ as $\sum_{r=k}^t X^{(r-1)e+k}_k P[p^k(q^{t+1} -q^r)/(q-1)]$ using Proposition \ref{prop_Xn0}.   This gives $n-k+1$ terms, of which all but the last term $X^n_k \cdots X^n_kP[p^n + p^{n-e} + \cdots + p^k]$ (with $p-1$ factors $X^n_k$) are shown to be $<_r X^n_k \cdots X^n_k$ (with $p$ factors) by expanding $P[(p^{n+e} - p^r)/(p-1)]$ as an iterated commutator.   The last term is equal to $P[p^n + p^{n-e} + \cdots + p^k]X^n_k \cdots X^n_k$ since $P[p^n + p^{n-e} + \cdots + p^k]$ commutes with $X^n_k$ in $E^0(\A_q)$.   We repeat this process by expanding the last factor $X^n_k$ using Proposition \ref{prop_Xn0}, and observe that all terms but the last are $<_r X^n_k \cdots X^n_k$ ($p$ factors), while the last term is equal to $P[p^n + p^{n-e} + \cdots + p^k]P[p^n + p^{n-e} + \cdots + p^k]X^n_k \cdots X^n_k$ ($p-2$ factors $X^n_k$).   Iterating this process a further $p-2$ times, we are left with the term $P[p^n + p^{n-e} + \cdots + p^k] \cdots P[p^n + p^{n-e} + \cdots + p^k]$ ($p$ factors), which is $0$ in $E^0(\A_q)$.
\end{proof}

This completes the proof of parts (i) and (ii) of Theorem \ref{th_Arnon_r}.   We turn to part (iii).

\begin{proposition} \label{prop_YZ_Mil}
The Y- and Z-Arnon A bases of $\A_q$ are triangular with respect to the Milnor basis, when the Arnon bases are taken in $\le_r$ order.
\end{proposition}

\begin{proof}
By Proposition \ref{prop_YZ_Pst}, it suffices to show that the Y- and Z-Arnon A bases are triangular with respect to the corresponding $P^s_t$ bases.   Again it suffices to work in $E^0(\A_q)$.   We consider the expression of $P^s_t$ basis elements in the Arnon A basis.

We begin by showing that $P^s_t = X^{s+(t-1)e}_s +$ $\le_r$-lower terms in the Arnon A basis.   The general case follows by taking products.   For example, $Sq(0,1) = Sq^2Sq^1 + Sq^1 \cdot Sq^2$ and $Sq(0,2) = Sq^4Sq^2 + Sq^2 \cdot Sq^4$ in the Arnon A basis.   Hence $Sq(0,1)Sq(0,2) = Sq^2Sq^1\cdot Sq^4Sq^2 + Sq^2Sq^1 \cdot Sq^2  \cdot Sq^4 + Sq^1 \cdot Sq^2 \cdot Sq^4Sq^2 + Sq^1 \cdot Sq^2 \cdot Sq^2  \cdot Sq^4$.   The first term on the right is the Arnon A basis element corresponding to the $P^s_t$ basis element $Sq(0,1)Sq(0,2)$, and the other terms are $\le_r$-lower.

As an example, when $p=2$ we have $P^s_2 = Sq[3 \cdot 2^s] = [Sq[2^{s+1}], Sq[2^s]] = Sq^{2^{s+1}}Sq^{2^s} + Sq^{2^s} \cdot Sq^{2^{s+1}}$ in the Arnon A basis for $E^0(\A_2)$.   In $\A_2$ itself, there are terms of higher May filtration, for example when $s=2$, $Sq(0,4) = Sq^8Sq^4 + Sq^4 \cdot Sq^8 + Sq(3,3)$, and $M(Sq(3,3)) = 6$.   In the general case, we can express $P^s_t$ as the iterated commutator
$$
P^s_t = P[p^s(q^t-1)/(q-1)] = [P[p^{s+(t-1)e}], [P[p^{s+(t-2)e}], [ \ldots [[P[p^{s+e}], P[p^s]] \ldots ]
$$
of length $t$ in the generators $P^{p^j}$.   The $\le_r$-maximal term in the expansion of the iterated commutator is the Arnon A element $X^{s+t-1}_s$.
\end{proof}

School of Mathematics,

Alan Turing Building,

University of Manchester,

Manchester M13 9PL, UK.

\end{document}